\documentclass{amsart}

\usepackage{amssymb,amsmath,amsthm,latexsym,amscd}
\usepackage{epsfig}
\usepackage{psfrag}
\usepackage{graphicx}
\usepackage{bbold}

\newtheorem{Theorem}{\bf Theorem}
\newtheorem{lemma}[Theorem]{\bf Lemma}
\newtheorem{proposition}[Theorem]{\bf Proposition}

\newtheorem{definition}[Theorem]{\bf Definition}
\newtheorem{example}[Theorem]{\bf Example}
\newtheorem{remark}[Theorem]{\bf Remark}
\newtheorem{theorem}[Theorem]{\bf Theorem}

\def\scfig #1 #2 {\resizebox{#2}{!}{\includegraphics{#1}}}

\newcommand{\be}{\begin{equation}}
\newcommand{\ee}{\end{equation}}

\def\hpic #1 #2 {\mbox{$\begin{array}[c]{l} 
\epsfig{file=#1,height=#2}\end{array}$}}
\def\wpic #1 #2 {\mbox{$\begin{array}[c]{l} 
\epsfig{file=#1,width=#2}\end{array}$}}

\begin{document}

\title[Planar algebras, quantum information theory and subfactors]{Planar algebras, quantum information theory and subfactors}

%\author{Sandipan De}
\author{Vijay Kodiyalam}
\author{Sruthymurali}
\author{V. S. Sunder}
\address{The Institute of Mathematical Sciences, Chennai, India and Homi Bhabha National Institute, Mumbai, India}
\email{vijay@imsc.res.in,sruthym@imsc.res.in,sunder@imsc.res.in}
\thanks{The authors are very grateful to Prof. Ajit Iqbal Singh for bringing the paper of Reutter and Vicary to their
attention.}
\subjclass[2010]{Primary 46L37, 81P45, 81P68}
%%\date{January 1, 1994 and, in revised form, June 22, 1994.}
%
%
%
%%\keywords{Semisimple and cosemisimple Hopf algebra, Drinfeld double, basic construction, planar algebra}
%
%
\begin{abstract} We define generalised notions of biunitary elements in planar algebras and show that objects arising in quantum information theory such as Hadamard matrices,
quantum latin squares and unitary error bases are all given by biunitary elements in the spin planar algebra. We show that there are natural subfactor planar algebras associated with biunitary elements.
\end{abstract}
\maketitle

\section{Introduction}
The motivation for this paper comes from the beautiful results of Reutter and Vicary in \cite{RttVcr2016} in which planar algebraic
constructions are used to treat a variety of objects in quantum information theory such as Hadamard matrices,
quantum latin squares and unitary error bases. While pictorial and planar algebraic techniques are used throughout 
that paper, no planar algebra actually makes an appearance, leading to the question as to where these
objects actually live.

In \S 2 we describe Jones' spin planar algebra and a specific recent presentation of it by generators and 
relations as discussed in \cite{KdyMrlSniSnd2019}.

We define and identify an equivalent formulation of the notion of a biunitary element in a planar algebra in \S 3  and show that the spin planar algebra is the natural receptacle of all of the following objects -
Hadamard matrices, quantum Latin squares, biunitary matrices (and unitary error bases) - by identifying these with appropriate types of biunitary
elements (and generalised versions of these) in the spin planar algebra.

The construction of subfactors from biunitary matrices is well known - see \cite{HgrSch1989} and \cite{Snd1989} - and our results naturally suggest that there might
be subfactors associated to our biunitary elements and we show in \S 4 that this is indeed the case - by constructing appropriate planar algebras. 

In the particular case of a 
Latin square arising from a group multiplication table, we identify this planar algebra, unsurprisingly,  with the very well understood
planar algebra of a group in the final \S 5.

\section{The spin planar algebra}

Let us recall Jones' spin planar algebra $P$ which was shown to have a presentation in terms of generators and relations in \cite{KdyMrlSniSnd2019} as follows. Let $S=\{s_1,\hdots,s_n\}$ be a finite set. Take the label set $L=L_{(0,-)}=S$
equipped with the identity involution $*$. Then $P = P(S)$ is defined to be the quotient $P=P(L,R)$ of the universal planar algebra $P(L)$ by the set $R$ of relations given in Figures \ref{fig:spinplanar1} and \ref{fig:spinplanar2}.

The facts that we will need about the spin planar algebra are summarised in the following two results.

%\begin{figure}[h]
%\centering
%\psfrag{= sqrtn}{\tiny $= \sqrt{n}$}
%\psfrag{s_i}{\tiny $s_i$}
%\psfrag{= frac{1}{sqrtn}}{\tiny $= \frac{1}{\sqrt{n}}$}
%\includegraphics[height=1.5cm]{spinplanar1.eps}
%\caption{The white and black modulus relations}
%\label{fig:spinplanar1}
%\end{figure}

\begin{figure}[!h]
\begin{center}
\psfrag{v+}{\huge $v_+$}
\psfrag{v-}{\huge $s_i$}
\psfrag{u1}{\Huge $\displaystyle{\frac{1}{\mu(v_+)}}$}
\psfrag{u2}{\Huge $\displaystyle{\frac{1}{\mu(v_-)}}$}
\psfrag{text1}{\Huge $= \displaystyle{\sqrt{n}}$}
\psfrag{text2}{\Huge $= \displaystyle{\frac{1}{\sqrt{n}}}$}
\resizebox{12.0cm}{!}{\includegraphics{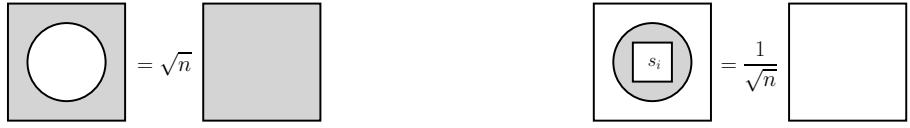}}
\end{center}
\caption{The white and black modulus relations}
\label{fig:spinplanar1}
\end{figure}

%\begin{figure}[h]
%\centering
%\psfrag{= sqrtn}{\tiny $= \sqrt{n}$}
%\psfrag{s_i}{\tiny $s_i$}
%\psfrag{s_j}{\tiny $s_j$}
%\psfrag{sum_i}{\tiny $\displaystyle{\sum_i}$}
%\psfrag{= delta_ij}{\tiny $= \delta_{i,j}$}
%\psfrag{= frac{1}{sqrtn}}{\tiny $= \frac{1}{\sqrt{n}}$}
%\includegraphics[height=1.5cm]{spinplanar2.eps}
%\caption{The multiplication relation and the black channel relations}
%\label{fig:spinplanar2}
%\end{figure}

\begin{figure}[!h]
\begin{center}
%\psfrag{zab}{\huge $\zeta a + b$}
%\psfrag{eq}{\huge $=$}
\psfrag{fi}{\Huge $s_i$}
\psfrag{fj}{\Huge $s_j$}
\psfrag{text2}{\Huge $\displaystyle{\sum\limits_i}$}
\psfrag{=}{\Huge $=\frac{1}{\sqrt{n}}$}
\psfrag{=dij}{\bf{\Huge $=\delta_{ij}$}}
\resizebox{11.0cm}{!}{\includegraphics{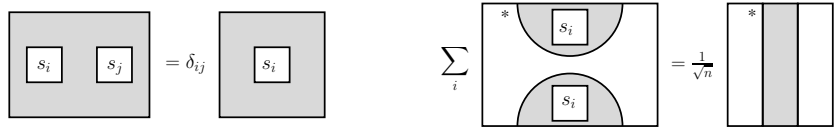}}
\end{center}
\caption{The multiplication relation and the black channel relation}
\label{fig:spinplanar2}
\end{figure}

%\begin{figure}[h]
%\centering
%\psfrag{= sqrtn}{\tiny $= \sqrt{n}$}
%\psfrag{s_i}{\tiny $s_i$}
%\psfrag{s_j}{\tiny $s_j$}
%\psfrag{= sum_i}{\tiny $= \displaystyle{\sum_i}$}
%\psfrag{= delta_ij}{\tiny $= \delta_{i,j}$}
%\psfrag{= frac{1}{sqrtn}}{\tiny $= \frac{1}{\sqrt{n}}$}
%\includegraphics[height=1.3cm]{unit2.eps}
%\caption{The unit and modulus relations}
%\label{fig:spinplanar3}
%\end{figure}

\begin{figure}[!h]
\begin{center}
\psfrag{v+}{\huge $v_+$}
\psfrag{v-}{\huge $s_i$}
\psfrag{u1}{\Huge $\displaystyle{\frac{1}{\mu(v_+)}}$}
\psfrag{u2}{\Huge $\displaystyle{\frac{1}{\mu(v_-)}}$}
\psfrag{xi}{\Huge $\xi,\xi$}
\psfrag{text3}{\Huge $= \displaystyle{\sqrt{n}}$}
\psfrag{text1}{\Huge $= \displaystyle{\sum\limits_{v_+ \in \mathcal{V}_+}}$}
\psfrag{text2}{\Huge $= \displaystyle{\sum\limits_{i}}$}
%\psfrag{text3}{\Huge $= \displaystyle{\sum\limits_{\xi: \mathcal{V}_+ \rightarrow \mathcal{V}_-}}$}
\resizebox{9cm}{!}{\includegraphics{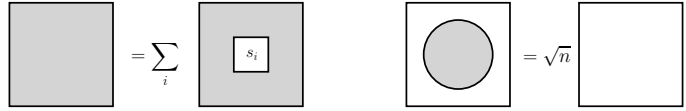}}
\end{center}
\caption{The unit and modulus relations}
\label{fig:spinplanar3}
\end{figure}

%\begin{figure}[h]
%\centering
%\psfrag{= sqrtn}{\tiny $= \sqrt{n}$}
%\psfrag{=}{$=$}
%\psfrag{s_i}{\tiny $s_i$}
%\psfrag{s_j}{\tiny $s_j$}
%\psfrag{sum_i}{\tiny $\displaystyle{\sum_i}$}
%\psfrag{= delta_ij}{\tiny $= \delta_{i,j}$}
%\psfrag{= frac{1}{sqrtn}}{\tiny $= \frac{1}{\sqrt{n}}$}
%\includegraphics[height=1.3cm]{spinplanar4.eps}
%\caption{Another unit relation}
%\label{fig:spinplanar4}
%\end{figure}

\begin{figure}[!h]
\begin{center}
%\psfrag{zab}{\huge $\zeta a + b$}
%\psfrag{eq}{\huge $=$}
\psfrag{fi}{\Huge $s_i$}
\psfrag{fj}{\Huge $s_j$}
\psfrag{text2}{\Huge $\displaystyle{\sum\limits_i}$}
\psfrag{=}{\Huge $=$}
\psfrag{=dij}{\bf{\Huge $=\delta_{ij}$}}
\resizebox{4cm}{!}{\includegraphics{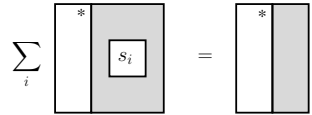}}
\end{center}
\caption{Another unit relation}
\label{fig:spinplanar4}
\end{figure}

\begin{lemma}[Lemma 2 of \cite{KdyMrlSniSnd2019}]
 The unit and modulus relations of Figures \ref{fig:spinplanar3} and \ref{fig:spinplanar4} hold in the planar algebra $P$.
\end{lemma}
%\begin{figure}[h]
%\centering
%\psfrag{= sqrtn}{\tiny $= \sqrt{n}$}
%\psfrag{s_i_1}{\scalebox{.5}{$s_{i_1}$}}
%\psfrag{s_i_2}{\scalebox{.5}{$s_{i_2}$}}
%\psfrag{s_i_m}{\scalebox{.5}{$s_{i_m}$}}
%\psfrag{s_j_1}{\scalebox{.5}{$s_{j_1}$}}
%\psfrag{s_j_2}{\scalebox{.5}{$s_{j_2}$}}
%\psfrag{s_j_m}{\scalebox{.5}{ $s_{j_m}$}}
%\psfrag{s_j_3}{\scalebox{.5}{$s_{j_3}$}}
%\psfrag{s_i_m-1}{\scalebox{.5}{ $s_{i_{m-1}}$}}
%\psfrag{hdots}{$\hdots$}
%% \psfrag{sum_i}{\tiny $\displaystyle{\sum_i}$}
%% \psfrag{= delta_ij}{\tiny $= \delta_{i,j}$}
%% \psfrag{= frac{1}{sqrtn}}{\tiny $= \frac{1}{\sqrt{n}}$}
%\includegraphics[height=4cm]{spinplanara5.eps}
%\caption{Bases $\mathcal{B}_{(2m,\pm)}$ for $P_{(2m,\pm)}$ for $m \geq 1$}
%\label{fig:spinplanar5}
%\end{figure}
%
%\begin{figure}[h]
%\centering
%\psfrag{= sqrtn}{\tiny $= \sqrt{n}$}
%\psfrag{s_i_1}{\scalebox{.5}{$s_{i_1}$}}
%\psfrag{s_i_2}{\scalebox{.5}{$s_{i_2}$}}
%\psfrag{s_i_m}{\scalebox{.5}{$s_{i_m}$}}
%\psfrag{s_i_m+1}{\scalebox{.5}{$s_{i_{m+1}}$}}
%
%\psfrag{s_j_1}{\scalebox{.5}{$s_{j_1}$}}
%\psfrag{s_j_2}{\scalebox{.5}{$s_{j_2}$}}
%\psfrag{s_j_m}{\scalebox{.5}{ $s_{j_m}$}}
%\psfrag{s_j_3}{\scalebox{.5}{$s_{j_3}$}}
%\psfrag{s_i_m-1}{\scalebox{.5}{ $s_{i_{{m\text{\tiny -}1}}}$}}
%\psfrag{hdots}{$\hdots$}
%% \psfrag{sum_i}{\tiny $\displaystyle{\sum_i}$}
%% \psfrag{= delta_ij}{\tiny $= \delta_{i,j}$}
%% \psfrag{= frac{1}{sqrtn}}{\tiny $= \frac{1}{\sqrt{n}}$}
%\includegraphics[height=4cm]{spinplanar6.eps}
%\caption{Bases $\mathcal{B}_{(2m+1,\pm)}$ for $P_{(2m+1,\pm)}$ for $m \geq 0$}
%\label{fig:spinplanar6}
%\end{figure}

\begin{theorem}[Theorem 1 of \cite{KdyMrlSniSnd2019}]{\label{spinthm}}
The spin planar algebra $P$ is a finite dimensional $C^*$- planar algebra with modulus $\sqrt{n}$ and such that $\text{dim}(P_{(0,+)})=1$, $\text{dim}(P_{(0,-)})=n$ and $\text{dim}(P_{(k,\pm)})=n^k$ for all $k>0$.  
Bases for $P_{(k,\pm)}$ for $k > 0$ are given as in Figures \ref{fig:spinplanar5} and \ref{fig:spinplanar6} for $k$ even and odd respectively while a basis of $\text{dim}(P_{(0,-)})$ is given by $S(i)$ which will denote the $(0,-)$-tangle with a single internal $(0,-)$ box labelled $s_i$ (which appears on the right hand sides of the multiplication relation of Figure \ref{fig:spinplanar2} or of the unit relation of Figure \ref{fig:spinplanar3}).
\end{theorem}

\begin{figure}[!h]
\begin{center}
\psfrag{si1}{\Huge $s_{i_1}$}
\psfrag{si2}{\Huge $s_{i_2}$}
\psfrag{sik}{\Huge $s_{i_m}$}
\psfrag{sj1}{\Huge $s_{j_1}$}
\psfrag{sj2}{\Huge $s_{j_2}$}
\psfrag{sjk}{\Huge $s_{j_m}$}
\psfrag{sp}{\Huge $s_{p}$}
\psfrag{sq}{\Huge $s_{q}$}
\psfrag{sim}{\Huge $s_{i_m}$}
\psfrag{sikm1}{\Huge $s_{i_{m-1}}$}
\psfrag{sjkm1}{\Huge $s_{j_{2}}$}
%\psfrag{zab}{\huge $\zeta a + b$}
%\psfrag{eq}{\huge $=$}
\psfrag{fi}{\Huge $s_i$}
\psfrag{fj}{\Huge $s_j$}
\psfrag{cdots}{\Huge $\displaystyle{\cdots}$}
\psfrag{=}{\Huge $=$}
\psfrag{=dij}{\bf{\Huge $=\delta_{ij}$}}
\resizebox{9.0cm}{!}{\includegraphics{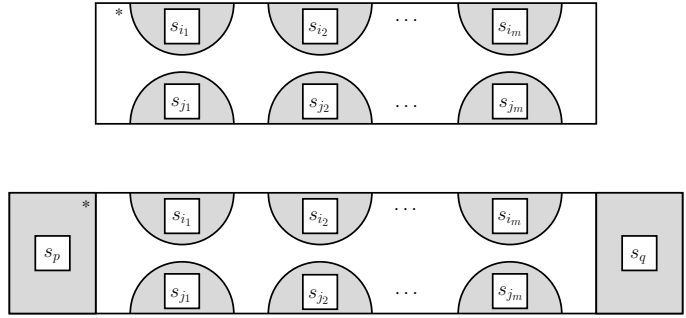}}
\end{center}
\caption{Bases ${\mathcal B}_{(2m,+)}$ for $m \geq 1$ and ${\mathcal B}_{(2m+2,-)}$ for $m \geq 0$}
\label{fig:spinplanar5}
\end{figure}

\begin{figure}[!h]
\begin{center}
\psfrag{si1}{\Huge $s_{i_1}$}
\psfrag{si2}{\Huge $s_{i_2}$}
\psfrag{sik}{\Huge $s_{i_m}$}
\psfrag{sj1}{\Huge $s_{j_1}$}
\psfrag{sj2}{\Huge $s_{j_2}$}
\psfrag{sjk}{\Huge $s_{q}$}
\psfrag{sp}{\Huge $s_{p}$}
\psfrag{sjmm1}{\Huge $s_{j_2}$}
\psfrag{sjm}{\Huge $s_{j_m}$}
\psfrag{sikm1}{\Huge $s_{i_{m}}$}
\psfrag{simp1}{\Huge $s_{j_{m}}$}
\psfrag{sjkm1}{\Huge $s_{j_{m}}$}
\psfrag{cdots}{\Huge $\displaystyle{\cdots}$}
%\psfrag{zab}{\huge $\zeta a + b$}
%\psfrag{eq}{\huge $=$}
\psfrag{fi}{\Huge $s_i$}
\psfrag{fj}{\Huge $s_j$}
\psfrag{cdots}{\Huge $\displaystyle{\cdots}$}
\psfrag{=}{\Huge $=$}
\psfrag{=dij}{\bf{\Huge $=\delta_{ij}$}}
\resizebox{8.0cm}{!}{\includegraphics{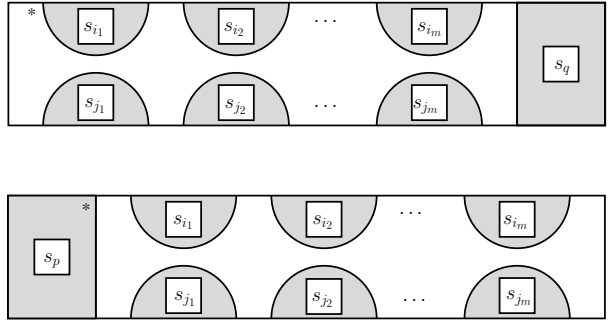}}
\end{center}
\caption{Bases ${\mathcal B}_{(2m+1,\pm)}$  for $m \geq 0$}
\label{fig:spinplanar6}
\end{figure}

%\begin{figure}[h]
%\centering
%\psfrag{s_i}{$s_i$}
%\includegraphics[height=1.5cm]{basis0minus.eps}
%\caption{Definition of $S(i)$}
%\label{fig:basis0minus}
%\end{figure}
%\begin{corollary}
%$\{S(i):i=1,\hdots,n\}$ is a basis for $P_{(0,-)}$.
%\end{corollary}
%\begin{corollary}
%The spin planar $P$ has modulus $\sqrt{n}$.
%
%\end{corollary}

\section{Biunitarity in planar algebras}

In this section we will first define the notion of biunitary element in a $*$- planar algebra. 
For each $k \in {\mathbb N}$, we have the rotation tangle $R(k,\epsilon) = R(k,\epsilon)_{(k,\epsilon)}^{(k,-\epsilon)}$ and its $\ell$-fold iteration  $R(k,\epsilon,\ell)$ given as in Figure \ref{fig:lrotation}.
\begin{figure}[h]
\centering
\psfrag{k-1}{\tiny $k-1$}
\psfrag{1}{\tiny $1$}
\psfrag{k-l}{\tiny $k-\ell$}
\psfrag{l}{\tiny $\ell$}
\psfrag{R}{$R(k,\epsilon)=R(k,\epsilon)_{(k,\epsilon)}^{(k,-\epsilon)}$}
\psfrag{R^l}{$R(k,\epsilon,\ell)=\displaystyle{R(k,\epsilon,\ell)_{(k,\epsilon)}^{(k,(-1)^\ell\epsilon)}}$}
%\psfrag{$R=R_{(k,+)}^{(k,-)}$}{$R(k)=R_{(k,+)}^{(k,-)}$}
%\psfrag{$R^l=R_{(k,+)}^{(k,-^l)}$}{$R(k)^\ell=\displaystyle{R(k)_{(k,+)}^{(k,(-)^\ell)}}$}
\includegraphics[height=3.5cm]{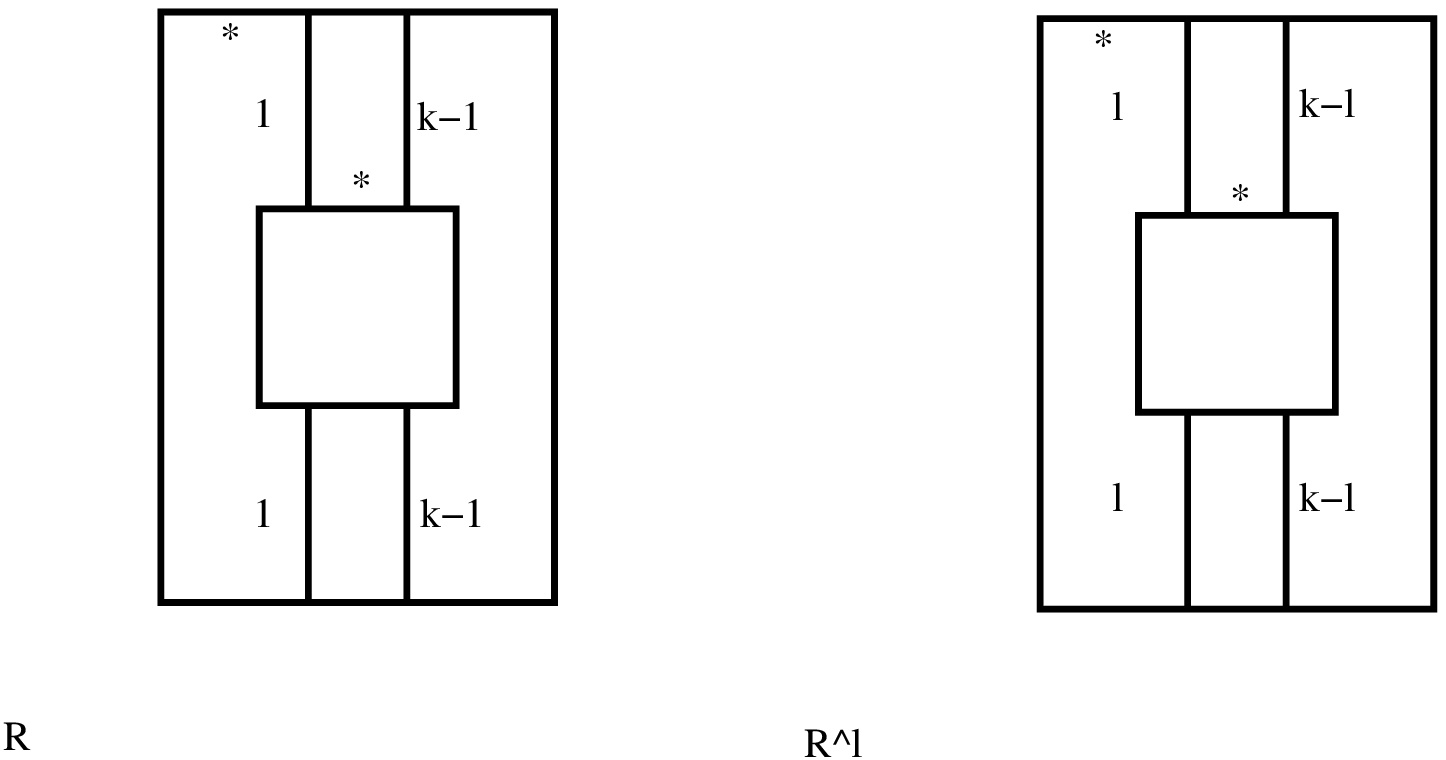}
\caption{Rotation tangles}
\label{fig:lrotation}
\end{figure}

In Figure \ref{fig:lrotation} and in the sequel we adopt two conventions: (i) $(-)^\ell$ denotes $\pm$ according to the parity of $\ell$, and (ii) In view of the difficulty of shading diagrams which depend on the parity of $\ell$, we will
dispense with shading figures since the shading is uniquely determined by the sub- and superscripts of the tangle.

\begin{definition}%[\textbf{$\{0,\ell\}$ biunitary}]
Let $P$ be a $*$-planar algebra and let $u \in P_{(k,\epsilon)}$. For $0 < \ell < k$, the element $u$ is said to be a $\{0,\ell\}$-biunitary element if the elements $u \in P_{(k,\epsilon)}$  and $Z_{R(k,\epsilon, \ell)}(u) \in P_{(k,\epsilon(-)^\ell)}$  are both unitary.
\end{definition}
%\begin{figure}[!htb]
%\psfrag{k-1}{\tiny $k-\ell$}
%\psfrag{1}{\tiny $\ell$}
%\psfrag{$R^{(k,+)}_{(k,-)}$: Rotation tangle}{$R_{(k,+)}^{(k,-)}$: Rotation tangle }
%\includegraphics[height=3.5cm]{rotationtangle.eps}
%\caption{Rotation tangle}
%\label{fig:rotation}
%\end{figure}
\begin{lemma}\label{lemma:eqcond}
The element $u \in P_{(k,\epsilon)}$ is $\{0,\ell\}$-biunitary if and only if the relations in Figure \ref{fig:unitary} hold in $P_{(k,\epsilon)}$.
\end{lemma}
%\begin{figure}[h]
%\centering
%\psfrag{k-l}{\tiny $k-\ell$}
%%{\scalebox{.5}
%%{$k-\ell$}}
%\psfrag{=}{$=$}
%\psfrag{l}{\scalebox{.5}{ $\ell$}}
%\psfrag{k}{\scalebox{.5}{ $k$}}
%\psfrag{u}{\scalebox{.8}{ $u$}}
%\psfrag{u^*}{\scalebox{.8}{$u^*$}}
%%\psfrag{(uu^*=1)}{$(uu^*=1)$}
%%\psfrag{(u^*u=1)}{$(u^*u=1)$}
%%\psfrag{(R^lu(R^lu)^*=1)}{$(Z_{R^\ell}( u)(Z_{R^\ell}( u))^*=1)$}
%%\psfrag{((R^lu)^*R^lu=1)}{$((Z_{R^\ell}( u))^*Z_{R^\ell}( u)=1)$}
%
%\psfrag{uu^*}{\small $uu^*$}
%\psfrag{u^*u}{\small$u^*u$}
%\psfrag{1}{$1$}
%\psfrag{X}{\tiny $Z_{R(k,-\ell)}(Z_{R(k,\ell)}(u)Z_{R(k,\ell)}(u)^*)$
%}
%\psfrag{Y}{\tiny $Z_{R(k,-\ell)}(1)$}
%\psfrag{Z}{\tiny $Z_{R(k,-\ell)}(Z_{R(k,\ell)}(u)^*Z_{R(k,\ell)}(u))$}
%\psfrag{1}{\tiny $1$}
%
%%\psfrag{$R^{(k,+)}_{(k,-)}$: Rotation tangle}{$R_{(k,+)}^{(k,-)}$: Rotation tangle }
%\includegraphics[height=6.5cm]{unitarylemma.eps}
%\caption{$\{0,\ell\}$-biunitarity relations}
%\label{fig:unitary}
%\end{figure}
\begin{figure}[h]
%\centering
%\psfrag{k-l}{\tiny $k-\ell$}
%%{\scalebox{.5}
%%{$k-\ell$}}
%\psfrag{=}{$=$}
%\psfrag{l}{\scalebox{.5}{ $\ell$}}
%\psfrag{k}{\scalebox{.5}{ $k$}}
%\psfrag{u}{\scalebox{.8}{ $u$}}
%\psfrag{u^*}{\scalebox{.8}{$u^*$}}
%%\psfrag{(uu^*=1)}{$(uu^*=1)$}
%%\psfrag{(u^*u=1)}{$(u^*u=1)$}
%%\psfrag{(R^lu(R^lu)^*=1)}{$(Z_{R^\ell}( u)(Z_{R^\ell}( u))^*=1)$}
%%\psfrag{((R^lu)^*R^lu=1)}{$((Z_{R^\ell}( u))^*Z_{R^\ell}( u)=1)$}
%
%\psfrag{uu^*}{\small $uu^*$}
%\psfrag{u^*u}{\small$u^*u$}
%\psfrag{1}{$1$}
%\psfrag{X}{\tiny $Z_{R(k,-\ell)}(Z_{R(k,\ell)}(u)Z_{R(k,\ell)}(u)^*)$
%}
%\psfrag{Y}{\tiny $Z_{R(k,-\ell)}(1)$}
%\psfrag{Z}{\tiny $Z_{R(k,-\ell)}(Z_{R(k,\ell)}(u)^*Z_{R(k,\ell)}(u))$}
%\psfrag{1}{\tiny $1$}
%
%%\psfrag{$R^{(k,+)}_{(k,-)}$: Rotation tangle}{$R_{(k,+)}^{(k,-)}$: Rotation tangle }
\psfrag{kml}{\tiny{$k-\ell$}}
\psfrag{l}{\tiny{$\ell$}}
\psfrag{u}{\small $u$}
\psfrag{u*}{\small $u^*$}
\psfrag{=}{\tiny{$=$}}
\includegraphics[height=6.5cm]{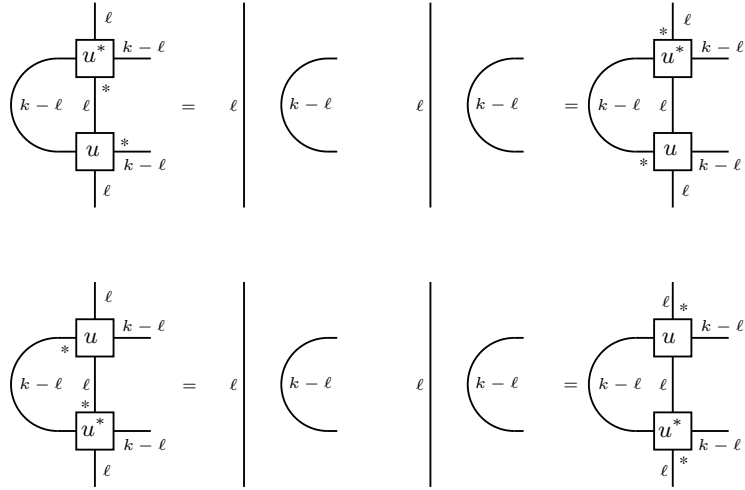}
\caption{$\{0,\ell\}$-biunitarity relations}
\label{fig:unitary}
\end{figure}

\begin{proof}
After choosing the external $*$-arc appropriately, the relations on top in Figure \ref{fig:unitary} are
equivalent to the unitarity of $u$ while the relations on the bottom are equivalent to the unitarity of 
$Z_{R(k,\epsilon, \ell)}(u)$. \end{proof}
%The first relation in Figure \ref{fig:unitary} is equivalent to the unitarity of $u$ while the second relation 
%is seen, by moving the external $*$ counterclockwise $\ell$ steps, to be equivalent to the unitarity of
%$Z_{R(k,\ell)}(u)$. \end{proof}

\begin{remark}
Observe that if $u \in P_{(k,\epsilon)}$ is a $\{0,\ell\}$-biunitary element, then so are $Z_{R(k,\epsilon,k)}(u) \in P_{(k,\epsilon(-)^k)}$ and
$Z_{R(k,\epsilon,-\ell)}(u^*) \in  P_{(k,\epsilon(-)^\ell)}$.
\end{remark}

In the rest of this section we will show that certain biunitary elements in the spin planar algebra are in one-to-one correspondence with some objects that arise in quantum information theory. We first define these objects.

\begin{definition}%[\textbf{Hadamard matrix}]
An $n \times n$ complex matrix $H = ((h_{ij}))$  is said to be a complex Hadamard matrix if $HH^* = nI$  and
 $|h_{ij}| = 1$ for each $i,j$.
\end{definition} 

\begin{example}
Let $\omega \in {\mathbb C}$ be a primitive $n^{th}$-root of unity. The matrix
%$H=\frac{1}{\sqrt{2}}\begin{bmatrix}
%1 & 1\\
%1& -1
%\end{bmatrix}$. More generally of size $n$, 
$$\displaystyle{H=\begin{bmatrix}
1 & 1 & \cdots & 1 \\
1 & \omega & \cdots & \omega^{n-1}\\
1 & \omega^2 & \cdots & (\omega^2)^{n-1}\\
\vdots & \vdots& \cdots & \vdots \\
1 & \omega^{n-1} & \cdots & (\omega^{n-1})^{n-1}
\end{bmatrix}}$$
is a complex Hadamard matrix which is a multiple of the so-called Fourier matrix.
\end{example}
\begin{definition}%[\textbf{Latin Squares}]
A Latin square is an $n \times  n$  array filled with n different symbols, each occurring exactly once in each row and  in each column.
\end{definition}
\begin{example}\label{latin}
The multiplication table of a finite group is a Latin square. The smallest example which is not (equivalent to one) of this type
is of size 5 and is given by:
%A latin square of order $3$ is $\begin{bmatrix}
%A & B & C\\
%C & A & B \\
%B & C & A
%\end{bmatrix}$. Clearly group table of a finite group is a latin square. Consider the following one latin square which is not from a group but from a Quasi- Group or rather a loop,
$$\begin{bmatrix}
 1 & 2 & 3 & 4 & 5 \\
 2 & 4 & 1 & 5 & 3 \\
 3 & 5 & 4 & 2 & 1 \\
 4 & 1 & 5 & 3 & 2 \\
 5 & 3 & 2 & 1 & 4
\end{bmatrix}.
$$ 
%In the order $5$ upto an equivalence relation there are only two latin squares. One corresponds to the group $\mathbb{Z}_5$ and the other the above one.
\end{example}
\begin{definition}%[\textbf{Quantum Latin Squares}]
A quantum Latin square of size $n$ is an $n \times n$ matrix of vectors in $\mathbb{C}^n$ such that each row and each column is an orthonormal basis for $\mathbb{C}^n$.
\end{definition}
\begin{example}\label{lstoqls}
Any Latin square gives a quantum Latin square in the following simple-minded way. 
Let $\{e_1, \hdots, e_n\}$ be the standard orthonormal basis  of $\mathbb{C}^n$. Consider the $5 \times 5$ Latin square in the Example \ref{latin}. It gives the following quantum Latin square:
$$\begin{bmatrix}
 e_1 & e_2 & e_3 & e_4 & e_5 \\
 e_2 & e_4 & e_1 & e_5 & e_3 \\
 e_3 & e_5 & e_4 & e_2 & e_1 \\
 e_4 & e_1 & e_5 & e_3 & e_2 \\
 e_5 & e_3 & e_2 & e_1 & e_4
\end{bmatrix}$$\\ 
For more on quantum Latin squares and  non-trivial examples see \cite{MstVcr2015}.
\end{example}
\begin{definition}%[\textbf{Biunitary Matrix}]
A matrix $U = ((u^{ij}_{kl})) \in M_{n^2}(\mathbb{C})$ (for $i,j,k,l \in \{1,\cdots,n\}$) is said to be a biunitary matrix if both $U$ and its block transpose, say $V = ((v^{ij}_{kl}))$, defined by $v^{ij}_{kl} = u^{kj}_{il}$, are unitary matrices.
\end{definition}
\begin{example}
For examples of  biunitary matrices of size 9 which are, in addition, permutation matrices, and their applications
 in subfactor theory see \cite{KrsSnd1996}.
\end{example}

\begin{definition}
A unitary error basis for $M_n({\mathbb C})$ is a collection of $n^2$ unitary matrices which form an orthonormal basis with respect to the normalised trace inner product given by $\displaystyle{\langle A|B \rangle= \frac{Tr(B^*A)}{n}}$.
\end{definition}

\begin{example} The matrices $\{U^iV^j: 1 \leq i,j \leq n\}$, where $U$ is the $n\times n$ Fourier matrix and $V$ is the permutation matrix corresponding to the cycle $(1~2~\cdots~n)$,  form a unitary error basis.
\end{example}

%%%%%%%%%%%%%%%%%%%%%%%%%%%%%%%%%%%%%%%%%%%%%%%%%%%%%%%%%%%%%%%%%%%%%%%%%%%%%%%%%%%%%%%%%%%%%%%%%%%%%%%%%%%%%%%%%%%%%%%%%%%%%%%%%%%%%%%%%%%%%%%%%%%%%%%%%%%%%%%%%%%%%%%%%%%%%%%%%%

We will now state and prove the main theorem of this section. 
Let $P = P(S)$ be the spin planar algebra where $S = \{s_1,\cdots,s_n\}$. We will need some notation for the normalised version of the bases for $P_{(k,\pm)}$ for $k=1,\hdots$ given in Figures \ref{fig:spinplanar5} and \ref{fig:spinplanar6}.
We denote $(\sqrt{n})^{m}$ times the elements on the top and bottom in Figure \ref{fig:spinplanar5} by $e^{i_1\cdots i_m}_{j_1\cdots j_m}$ and $e[p)^{i_1 \cdots i_{m}}_{j_1 \cdots j_m}(q]$ respectively. Similarly, we denote $(\sqrt{n})^{m}$ times the
elements on the top and bottom in Figure \ref{fig:spinplanar6} by $e^{i_1\cdots i_m}_{j_1\cdots j_m}(q]$ and $e[p)^{i_1 \cdots i_{m}}_{j_1 \cdots j_m}$. In the following proof we will implicitly use the multiplication relations among these elements as in Lemma~10 of \cite{KdyMrlSniSnd2019}. We will also require the
action of the rotation tangle on the bases as stated in Lemma \ref{lemma:rotation} below (in the proof of Equations \ref{uunitary1}, \ref{br2} and \ref{r2uunitary3}).

\begin{lemma}\label{lemma:rotation} With notation as above,
\begin{eqnarray*}
Z_{R(2m,+)}(e^{i_1\cdots i_m}_{j_1\cdots j_m}) &=& \sqrt{n}~ e[j_1)^{i_1 \cdots i_{m-1}}_{j_2 \cdots j_m}(i_m] \\
Z_{R(2m+2,-)}(e[p)^{i_1 \cdots i_{m}}_{j_1 \cdots j_m}(q]) &=& \frac{1}{\sqrt{n}}~ e^{p\ \!\! i_1\cdots i_m}_{j_1\cdots j_m q}\\
Z_{R(2m+1,+)}(e^{i_1\cdots i_m}_{j_1\cdots j_m}(q]) &=& e[j_1)^{i_1\  \cdots \ i_{m}}_{j_2 \cdots j_m q}\\
Z_{R(2m+1,-)}(e[p)^{i_1 \cdots i_{m}}_{j_1 \cdots j_m}) &=& e^{p i_1\cdots i_{m-1}}_{j_1\ \cdots \ j_m}(i_m]\\
\end{eqnarray*}
\end{lemma}

% see corollary $11$ and notations given in remark above lemma $10$ in \cite{KdyMrlSniSnd2018}. %described in \cite{vijay2017planar}. 
\begin{theorem}{\label{spinquantumthm}}
There are natural 1-1 correspondences between the following sets:\\
1) $\{0,1\}$-biunitary elements in $P_{(2,+)}$ and Hadamard matrices of size $n \times n$,\\
2) $\{0,1\}$-biunitary elements in  $P_{(3,+)}$ and quantum Latin squares of size $n \times n $, and\\
3) $\{0,2\}$-biunitary elements in $P_{(4,+)}$ and biunitary matrices of size $n^2 \times n^2$.\\
%
%
%
%
%$\{0,1\}$ biunitaries in $P_{(2,+)}$, $P_{(3,+)}$ and $\{0,2\}$ biunitaries in $P_{(4,+)}$ are in one to one correspondence with Hadamard matrix of size $n \times n$, 
%Quantum Latin square of size $n \times n $ and
%Biunitary matrix of size $n^2 \times n^2$ respectively.
\end{theorem}
\begin{proof}
1) Let $\displaystyle{u=\sum_{i,j} a^i_j~~e^i_j\in P_{(2,+)}}$.
% be a $\{0,1\}$-biunitary element. 
 Then, 
\begin{equation}\label{eqn1}
uu^*=1 \Leftrightarrow \displaystyle{\sum_j a^i_j~~\overline{a^p_j}=\delta_{i,p}} ~~\forall~ i, p=1,\hdots,n,
\end{equation}
and 
\begin{equation}\label{uunitary1}
Z_{R(2,+)}(u)(Z_{R(2,+)}(u))^*=1 \Leftrightarrow n~a^i_j~~\overline{a^i_j}=1 ~~\forall~ i, j=1,\hdots,n.
\end{equation}
%$Z_R(u)(Z_{R}(u))^*=1$ if and only if
%%\begin{equation}\nonumber
%$n~a^i_j~~\overline{a^i_j}=1$.
Let $H = ((\sqrt{n} a^i_j))$. Then Equation \ref{eqn1} is equivalent to $HH^* = n~I$ and Equation \ref{uunitary1} is equivalent to $|\sqrt{n} a^i_j| = 1$.
Thus the association of $u$ with $H$ is a 1-1 correspondence between $\{0,1\}$-biunitary elements of $P_{(2,+)}$ and Hadamard matrices of size $n \times n$.\\
%
%
%
%
%
%\end{equation}
%or 
%\begin{equation}\label{ruunitary1}
%|a^i_j|=\frac{1}{\sqrt{n}}~~\forall~~i,j~=1,\hdots,n.
%\end{equation}
%Now let $H=((\sqrt{n} a^i_j))$. Then Equation \ref{uunitary1} implies that $HH^* = n~I$ and Equation \ref{ruunitary1} implies that each $\sqrt{n} a^i_j$ has modulus $1$. This proves that from a $\{0,1\}$ biunitary
%we have obtained an Hadamard matrix $H$. Now start from an Hadamard matrix say $H=[h^i_j]_{n \times n}$. Define  $\displaystyle{u=\sum_{i,j} h^i_j~~e^i_j\in P_{(2,+)}}$. Since $H$ is an Hadamard matrix, equations \ref{uunitary1} and
%\ref{ruunitary1} holds and this implies that $u$ is a $\{0,1\}$ biunitary. Hence the correspondence follows.\\

2) Now let $\displaystyle{u=\sum_{i,j,k}a_{ij}^k~~e^i_k(j] \in P_{(3,+)}}$. Then, 
\begin{equation}\label{uunitary2}
uu^*=1 \Leftrightarrow \displaystyle{\sum_k~a_{ij}^k~~\overline{a_{pj}^k}=\delta_{i,p},~~\forall ~~j,i,p=1,\hdots, n},
\end{equation}
and 
\begin{equation}\label{br2}
Z_{R(3,+)}(u)(Z_{R(3,+)}(u))^*=1 \Leftrightarrow \displaystyle{\sum_j~a_{ij}^k~~\overline{a_{pj}^k}=\delta_{i,p},~~\forall ~~k,i,p=1,\hdots, n}.
\end{equation}
Let $Q = ((Q^i_j))$ where $Q^i_j=(a_{1j}^i,a_{2j}^i,\hdots,a_{nj}^i) \in {\mathbb C}^n$. Then, 
Equation \ref{uunitary2} is equivalent to the column vectors of $Q$ forming an orthonormal basis for $\mathbb{C}^{n}$ and 
Equation \ref{br2} is equivalent to the row vectors of $Q$ forming an orthonormal basis for $\mathbb{C}^{n}$.
Thus the association of $u$ with $Q$ is a 1-1 correspondence between $\{0,1\}$-biunitary elements of $P_{(3,+)}$ and quantum Latin squares of size $n \times n$.\\
%
%
%
%
%
%
%
%Consider the matrix $Q=[Q_{ij}]$ where $Q_{ij}=(a_{1j}^i,a_{2j}^i,\hdots,a_{nj}^i)$. Then the equation \ref{uunitary2} implies that column vectors of $Q$ will form an orthonormal basis for $\mathbb{C}^{n}$ and the
%equation \ref{ruunitary2} implies that the row vectors of $Q$ form an orthonormal basis for $\mathbb{C}^{n}$. That is $Q$ is a Quantum Latin Square. Now all the steps are completely reversible implying the one to one correspondence.\\

3) Let $\displaystyle{u= \sum_{i,j,k,\ell}a^{ij}_{k\ell}e^{ij}_{\ell  k} \in P_{(4,+)}}$. Then,
\begin{equation}\label{uunitary3}
uu^*=1 \Leftrightarrow \displaystyle{\sum_{k,\ell}~~a^{ij}_{k\ell}~~\overline{a^{pq}_{k\ell}}=\delta_{i,p}~~\delta_{j,q}},~~\forall~~i,j,p,q=1,\hdots,n
\end{equation}
and 
\begin{equation}\label{r2uunitary3}
Z_{R(4,+,2)}(u)(Z_{R(4,+,2)}(u))^*=1 \Leftrightarrow \displaystyle{\sum_{j,k}~~a^{ij}_{k\ell}~~\overline{a^{pj}_{ks}}=\delta_{i,p}~~\delta_{\ell,s}},~~\forall~~i,p,\ell,s=1,\hdots,n
\end{equation}
Here let $U = ((a^{ij}_{k\ell}))$.
Then, 
Equation \ref{uunitary3} is equivalent to $U$ being unitary and 
Equation \ref{r2uunitary3} is equivalent to the block transpose of $U$ being unitary.
Thus the association of $u$ with $U$ is a 1-1 correspondence between $\{0,2\}$-biunitary elements of $P_{(4,+)}$ and biunitary matrices of size $n^2 \times n^2$.
\end{proof}
%
%Now let $A=[a^{ij}_{k\ell}]_{n^2 \times n^2}$ be an $n^2\times  n^2$ matrix with rows indexed by $(ij)$, $i,j=1,\hdots,n$ and columns indexed by $(k\ell)$, $k,\ell=1,\hdots,n$. Then the equation \ref{uunitary3} implies that
%$A$ is unitary and the equation \ref{r2uunitary3} implies that block transpose of $A$ is unitary and hence $A$ is biunitary. Reversing all the steps yields the one to one correspondence.
%\end{proof}

An analogous 1-1 correspondence result for unitary error bases requires a modified version of the notion of biunitary element which we will now define. Recall that a labelled annular tangle in a planar algebra $P$ is a tangle $A$ all of whose internal boxes except for one have been labelled by elements of the appropriate $P_{(k, \epsilon)}$'s. Actually we would like to consider linear extensions of the definition and of the vector
operations to linear combinations of such tangles, provided of course that all the annular tangles involved yield maps between the same spaces. We will use the term modified annular tangle for such  linear combinations.

%We will also call a linear combination of such, a labelled annular tangle.

%\begin{definition}[\textbf{Labelled Annular tangle}]
%A labelled annular tangle in any planar algebra $P$ means a tangle $A$ all of whose boxes except for one have been labelled by the appropriate $P_{(k, \epsilon)}$.
%\end{definition}
%
%
%\begin{example}
%Identity tangles, Rotation tanles and their powers(Should mention or not) since they have only one internal box. (check whether any extra example needed). We will also call a linear combination of such a labelled annular tangle.
%\end{example}

\begin{definition}
Let $P$ be a $*$- planar algebra and  $A,B$ be modified annular tangles with their unlabelled box  of colour $(k,\epsilon)$. An element $u \in P_{(k,\epsilon)}$ 
is said to be an $\{A,B\}$-biunitary element if $Z_A(u)$ and $Z_B(u)$ %\in P_{(m,\pm)}$ 
are both unitary.
\end{definition}

\begin{remark}
Note that a $\{0,\ell\}$-biunitary element in $P_{(k,\epsilon)}$ is nothing but an $\{I,R(k,\epsilon,\ell)\}$-biunitary element, where $I$ denotes the identity tangle of colour $(k,\epsilon)$.
\end{remark}

%\begin{definition}[\textbf{Unitary Error Basis}]
%An Unitary Error Basis(UEB) is a collection of $n^2$ complex matrices which form an orthonormal basis for $\mathbb{C}^{n^2}$ wrt the innerproduct $\displaystyle{\langle A|B \rangle= \frac{Tr(A^*B)}{n}}$.
%\end{definition}

In order to state the analogue of Theorem \ref{spinquantumthm} for unitary error bases we will need the modified annular tangle $A = A_{(4,+)}^{(4,+)}$, with labelled internal boxes coming from the  spin planar algebra $P$, defined by Figure \ref{fig:uebpicture}.
%\begin{figure}[h]
%\psfrag{sum_{u,v}n}{$\displaystyle{\sum_{u,v \in S}n}$}
%\psfrag{u}{$\tiny u$}
%\psfrag{v}{$\tiny v$}
%\includegraphics[height=3.5cm]{uebpicture.eps}
%\caption{The labelled annular tangle $A$}
%\label{fig:uebtangle}
%\end{figure}
\begin{figure}[h]
\centering
\psfrag{si}{\tiny $s_i$}
\psfrag{sj}{\tiny $s_j$}
\psfrag{n}{$n$}
\psfrag{sum_ij}{\large $\displaystyle{\sum_{i,j=1}^n}$}
\psfrag{*}{\tiny $*$}
\includegraphics[height=3.5cm]{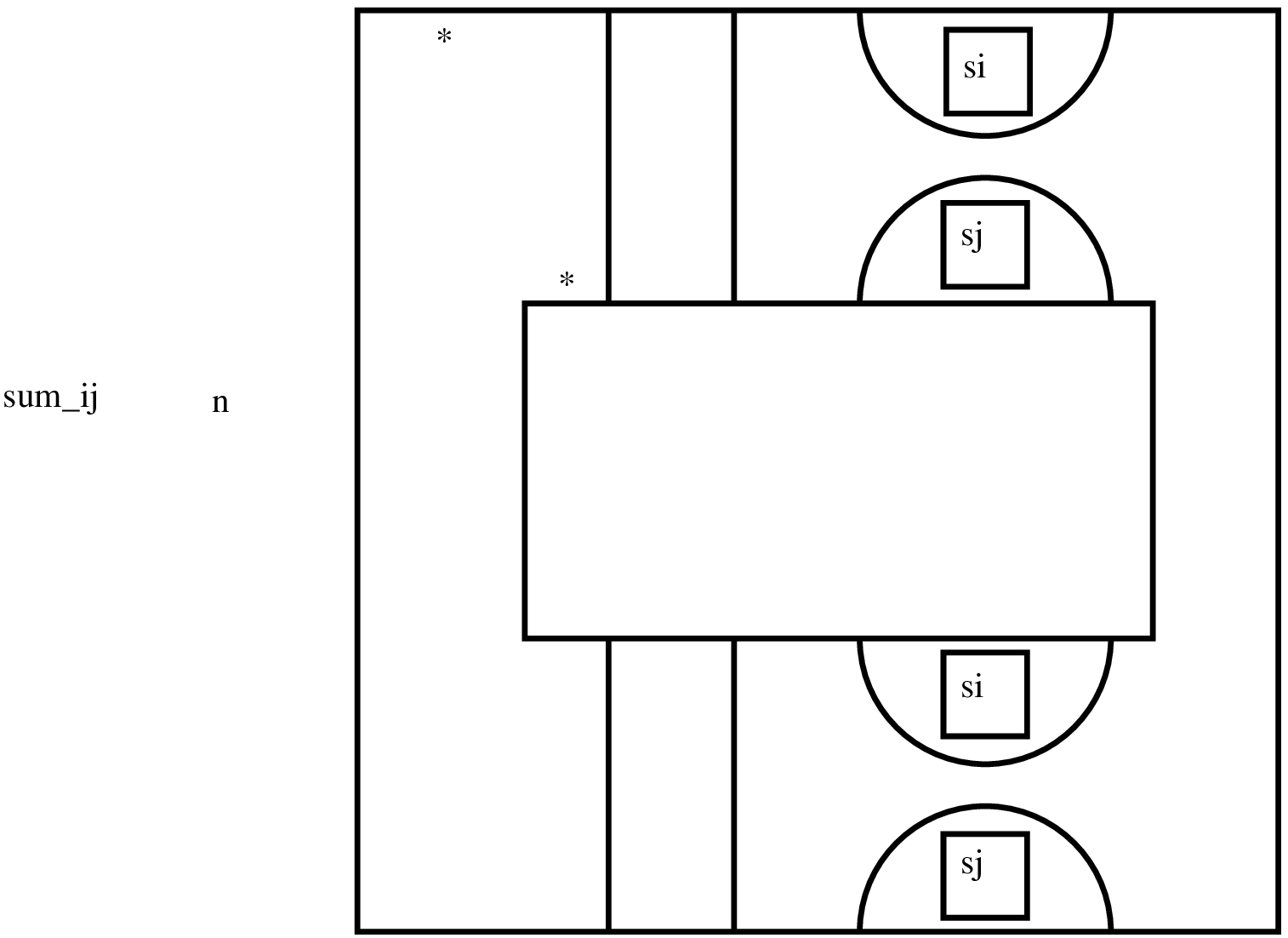}
\caption{The annular tangle $A$}
\label{fig:uebpicture}
\end{figure}

\noindent
Observe that $Z_A(e^{ij}_{kl}) = e^{il}_{kj}$.

\begin{proposition}
There is a natural 1-1 correspondence between $\{A,R(4,+)\}$-biunitary elements in $P_{(4,+)}$ and unitary error bases  in $M_n({\mathbb C})$.
%Consider the spin planar algebra $P$, see \cite{vijay2017planar}. Let $A$ be the annular tangle given by
%%\begin{figure}[h]
%%\psfrag{sum_{u,v}n}{$\displaystyle{\sum_{u,v \in S}n}$}
%%\psfrag{u}{$\tiny u$}
%%\psfrag{v}{$\tiny v$}
%%\includegraphics[height=3.5cm]{uebtangle.eps}
%%%\caption{}
%%\label{fig:uebtangle}
%%\end{figure}
%and $R$ be the rotation tangle. Then the $(A,R)$ biunitaries in $P_{(4,+)}$ are in one to one correspondence between the UEB's.
\end{proposition}

\begin{proof}
Let $\displaystyle{u=\sum_{i,j,k,\ell}~~a^{ij}_{k\ell}~~e^{ij}_{\ell k}~~ } \in P_{(4,+)}$. Then,
\begin{eqnarray*}
Z_{R(4,+)}(u) &=& \sum_{i,j,k,\ell}~~\sqrt{n}~~a^{ij}_{k\ell}~~e[\ell)^i_k(j], ~{\text {and}}\\
Z_A(u) &=& \sum_{i,j,k,\ell}~~a^{ij}_{k\ell}~~e^{ik}_{\ell j}.
\end{eqnarray*}
by Lemma \ref{lemma:rotation} and the observation above. Hence
\begin{equation}\label{auunitary}
(Z_A(u))^*Z_A(u)=1  \Leftrightarrow \sum_{i,k}~~a^{ij}_{k\ell}~~\overline{a^{iq}_{ks}}=\delta_{\ell s}~~\delta_{jq}~~\forall~~\ell, s,j,q=1,\hdots,n
\end{equation}
and
\begin{equation}\label{ruunitary}
Z_{R(4,+)}(u)(Z_{R(4,+)}(u))^*=1 \Leftrightarrow \sum_k~~a^{ij}_{k \ell }~~\overline{a^{pj}_{k \ell }}=\frac{1}{n}~~\delta_{ip}~~\forall j,\ell,i,p=1,\hdots,n
\end{equation} 

%We have,\\ 
%$\displaystyle{Z_R(u)=\sum_{i,j,k,\ell}~~\sqrt{n}~~a^{ij}_{k\ell}~~e(\ell)^i_k(j)}$ and $\displaystyle{Z_A(u)=\sum_{i,j,k,\ell}~~a^{ij}_{k\ell}~~e^{ik}_{\ell j}}$.
%Now the relations,\\ $Z_A(u)(Z_A(u))^*=1$, $(Z_A(u))^*Z_A(u)=1$, $Z_R(u)(Z_R(u))^*=1$ and $(Z_R(u))^*Z_R(u)=1$ imply the following four equations,
%
%\begin{equation}\nonumber
%\displaystyle{\sum_{j,\ell}~~a^{ij}_{k\ell}~~\overline{a^{pj}_{r\ell}}=\delta_{k,r}~~\delta_{i,p}}~~\forall~~k,r,i,p=1,\hdots,n
%\end{equation}
%
%\begin{equation}\nonumber
%\displaystyle{\sum_{i,k}~~a^{ij}_{k\ell}~~\overline{a^{iq}_{ks}}=\delta_{\ell s}~~\delta_{j,q}}~~\forall~~\ell, s,j,q=1,\hdots,n
%\end{equation}
%
%\begin{equation}\nonumber
%\displaystyle{\sum_k~~a^{ij}_{k \ell }~~\overline{a^{pj}_{k \ell }}=\frac{1}{n}~~\delta_{i,p}}~~\forall \ell,i,p=1,\hdots,n
%\end{equation}
%
%\begin{equation}\nonumber
%\displaystyle{\sum_i~~a^{iq}_{k s }~~\overline{a^{iq}_{rs }}=\frac{1}{n}~~\delta_{k,r}}~~\forall k,r,s=1,\cdots,n
%\end{equation}
%%\displaystyle{\sum_k~~a^{ij}_{k\ell}~~\overline{a^{pj}_{k\ell}}=\frac{1}{n}\delta_{i,p}}~~\forall \ell,i,p=1,\cdots,n

Let $B(j,\ell)$ be the $n \times n$ matrix given by $B(j,\ell) = ((\sqrt{n}a^{ij}_{k\ell}))$.  Then Equation \ref{ruunitary} is equivalent to $B(j,\ell)$ being a unitary matrix for all $j,\ell$ and
Equation \ref{auunitary} is equivalent to the collection $\{B(j,\ell)\}_{j,\ell}$ forming an orthonormal basis for $M_n({\mathbb C})$.
Thus the association of $u$ with $\{B(j,\ell)\}_{j,\ell}$ is a 1-1 correspondence between $\{A,R(4,+)\}$-biunitary elements of $P_{(4,+)}$ and unitary error bases of $M_n({\mathbb C})$.
%
%
%
%
%Then the above equations become
%
%\begin{equation}\label{auunitary4}
%\displaystyle{\sum_{j,\ell}~~b^{ij}_{k \ell }~~\overline{b^{pj}_{r \ell }}=n~~\delta_{k,r}~~\delta_{i,p}}~~\forall~~k,r,i,p=1,\hdots,n 
%\end{equation}
%
%\begin{equation}\label{au*unitary4}
%\displaystyle{\sum_{i,k}~b^{ij}_{k\ell}~~\overline{b^{iq}_{ks}}=~~n~~\delta_{\ell s}~~\delta_{j,q}}~~\forall~~\ell, s,j,q=1,\hdots,n 
%\end{equation}
%
%\begin{equation}\label{ruunitary4}
%\displaystyle{\sum_k~~b^{ij}_{k \ell  }~~\overline{b^{pj}_{ k \ell  }}=\delta_{i,p}}~~\forall \ell,i,p=1,\hdots,n
%\end{equation}
%
%\begin{equation}\label{ru*unitary}
%\displaystyle{\sum_i~~b^{iq}_{ks}~~\overline{b^{iq}_{rs}}=~~\delta_{k,r}}~~\forall k,r,s=1,\hdots,n 
%\end{equation}
%
%Now define the $n^2$ unitary matrices of size $n \times n$ as follows.
%For $j,\ell=1,\hdots,n$, let $\displaystyle{B_{\tiny {j\ell}}=[b^{\tiny{ij}}_{\tiny{k \ell }}]_{n \times n}}$. Then the equation \ref{ruunitary4} implies that each $B_{j\ell}$ are unitary matrices and the equation 
%\ref{au*unitary4} implies that
%the collection $\{B_{j\ell}: j,\ell =1,\hdots,n\}$ will form an orthonormal basis for for $\mathbb{C}^{n^2}$ wrt the innerproduct $\displaystyle{\langle A|B \rangle= \frac{Tr(A^*B)}{n}}$.
\end{proof}

\section{From biunitary elements to subfactor planar algebras}

Throughout this section, $P$ will be a spherical $C^*$-planar algebra with modulus $\delta$ and $u \in P_{(k,\epsilon)}$ will be a  $\{0,\ell\}$-biunitary. 
To this data, we will associate a $C^*$-planar subalgebra $Q$ of the $(\ell,\epsilon)^{th}$ cabling $^{(\ell,\epsilon)} P$ of $P$. The notion of cabling that we use here is a generalised version of the one 
defined in \cite{DeKdy2018}, so we give a careful definition.

\begin{definition}
Let $(\ell,\epsilon) \in {\mathbb N}\times \{\pm\}$. Define $(n,\eta)^{(\ell,\epsilon)}$ to be $(n\ell,\epsilon\eta^\ell)$.
Next, define the $(\ell,\epsilon)$-cable of a tangle $T$, denoted by $T^{(\ell,\epsilon)}$, as follows. Consider the tangle $T$ ignoring its shading and replace each of its strands (including the closed loops) by a cable of $\ell$ parallel strands without changing the $*$-arcs.  Introduce shading in this picture such that an $(n,\eta)$-box of $T$
becomes an $(n,\eta)^{(\ell,\epsilon)}$-box of $T^{(\ell,\epsilon)}$. 
\end{definition}

We omit the verification that this extends uniquely to a chequerboard shading of $T^{(\ell,\epsilon)}$ making it a tangle and that $T \mapsto T^{(\ell,\epsilon)}$ is an `operation on tangles' in the sense of \cite{KdySnd2004}. The corresponding operation on planar algebras will be denoted by $P \mapsto ^{(\ell,\epsilon)}\!\!\!P$. To give an example of cabling, note that the $(\ell,\epsilon)$-cable of the rotation tangle $R(n,\eta)$ of Figure \ref{fig:lrotation} is given as in Figure \ref{fig:cabrot} below. The shading of the
$*$-arc of the external box is given by $\epsilon(-\eta)^\ell$.
\begin{figure}[h]
\centering
\psfrag{k-1}{\tiny $(n-1)\ell$}
\psfrag{1}{\tiny $\ell$}
%\psfrag{k-l}{\tiny $k-\ell$}
%\psfrag{l}{\tiny $\ell$}
%\psfrag{R}{$R(k)=R_{(k,+)}^{(k,-)}$}
%\psfrag{R^l}{$R(k,\ell)=\displaystyle{R(k)_{(k,+)}^{(k,(-)^\ell)}}$}
%\psfrag{$R=R_{(k,+)}^{(k,-)}$}{$R(k)=R_{(k,+)}^{(k,-)}$}
%\psfrag{$R^l=R_{(k,+)}^{(k,-^l)}$}{$R(k)^\ell=\displaystyle{R(k)_{(k,+)}^{(k,(-)^\ell)}}$}
\includegraphics[height=3.5cm]{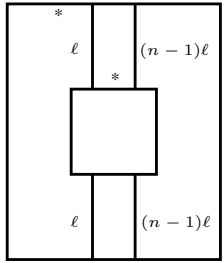}
\caption{The $(\ell,\epsilon)$-cabling of the rotation tangle $R(n,\eta)$}
\label{fig:cabrot}
\end{figure}

\begin{definition} For $(\ell,\epsilon) \in {\mathbb N}\times \{\pm\}$, the planar algebra $^{(\ell,\epsilon)} P$ has underlying vector spaces given by $^{(\ell,\epsilon)} P_{(n,\eta)} = P_{(n,\eta)^{(\ell,\epsilon)}}$ with the tangle action given by 
$Z_T^{^{(\ell,\epsilon)} P} = Z_{T^{(\ell,\epsilon)}}^P$.
\end{definition}

%
%Explicitly, the vector spaces underlying $^{(m)} P$ are given by $(^{(m)} P)_{(k,\epsilon)} = P_{(k,\epsilon)^{(m)}}$ and the tangle action on $^{(m)} P$ is defined by $Z_T^{^{(m)} P} = Z_{T^{(m)}}^P$.
%

Before proceeding to define subspaces $Q_{(n,\eta)}$ of $^{(\ell,\epsilon)} P_{(n,\eta)}$, 
we begin with the inevitable notation. For $n=0,1,\hdots$, define the elements $u_{(n,\eta)} \in P_{(n\ell+k-\ell, \epsilon\eta^\ell)}$ as in Figure \ref{fig:definingtangle} - where the label in the last box is $u$ or $u^*$ depending on the parity of $n$.

%we restate Lemma \ref{lemma:eqcond} in a form that will be convenient in various proofs.
%
%\begin{lemma}\ref{eqcond2}
%The element $u \in P_{(k,\epsilon)}$ is $\{0,\ell\}$-biunitary if and only if the relations in Figure \ref{fig:unitary2} hold in $P_{(k,\epsilon)}$.
%
%\end{lemma}
%
%We first define the subspaces $Q_{(k,\pm)}$ and then verify that $Q$ is  a $C^*$-planar algebra which is a subfactor planar algebra in `good' cases.
%
%First the inevitable notation. For $n=0,1,\hdots$ define the elements $u_{(n,\pm)} \in P_{(n\ell+k-\ell, (\pm)^\ell)}$ as in Figure \ref{fig:definingtangle} - where the label in the last box is $u$ or $u^*$ depending on the parity of $n$.

%In this section we will construct planar subalgebras using the $\{0,\ell\}$ biunitaries. Let $P$ be a $*$ - planar algebra and $u \in P_{(k,+)}$ be a  $\{0,\ell\}$ biunitary. 
%
%
%
%
%
%
%For $n=0,1,\hdots$ define the elements $u_{(n,\pm)} \in P_{(n\ell+k-\ell, (\pm)^\ell)}$ as follows.

\begin{figure}[h]
\centering
\psfrag{u_{(0,pm)}}{$u_{(0,\pm)}$}
\psfrag{u_{(n,-)}:n>0}{$u_{(n,-)}:n > 0$}
\psfrag{u_{(n,+)}:n>0}{$u_{(n,+)}:n > 0$}
\psfrag{l}{\tiny $\ell$}
\psfrag{k-l}{\tiny $k-\ell$}
\psfrag{u}{\tiny $u$}
\psfrag{u^*}{\tiny ${u^*}$}
\psfrag{u/u^*}{\tiny $u~~\slash~ ~u^*$}
\psfrag{hdots}{$\dots$}
%\psfrag{/}{$/$}
\includegraphics[height=4.5cm]{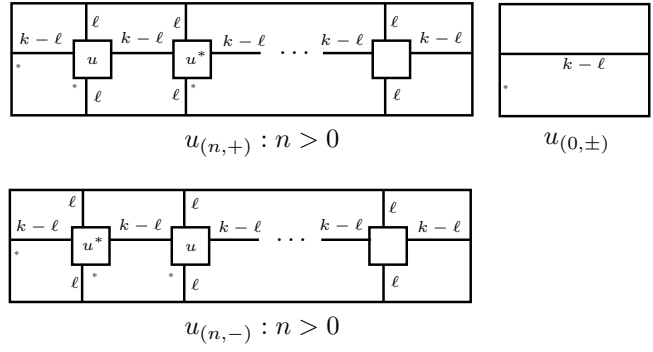}
\caption{The elements $u_{(n,\pm)}$}
\label{fig:definingtangle}
\end{figure}

%The elements $u_{(n,\eta)}$ satisfy the following useful relations.
%\begin{lemma}\label{lemma:circles}
%For each $(n,\eta) \in ({\mathbb N} \cup \{0\}) \times \{\pm\}$, the relations in Figure \ref{fig:unetarel} hold in $P$.
%\begin{figure}[h]
%\psfrag{une}{$u_{(n,\eta)}$}
%\psfrag{une*}{$u_{(n,\eta)}^*$}
%\psfrag{d}{$\delta^{k-\ell}$}
%\psfrag{nl}{\tiny $n\ell$}
%\psfrag{kml}{\tiny $k-\ell$}
%\includegraphics[height=3cm]{unetarel.eps}
%\caption{Relations satisfied by $u_{(n,\eta)}$}
%\label{fig:unetarel}
%\end{figure}
%\end{lemma}
%
%\begin{proof}
%The proof is an easy induction on $n$, using all the equalities in Lemma \ref{lemma:eqcond} and the existence of modulus for $P$.
%\end{proof}

\begin{proposition}\label{xycircles}
For $(n,\eta) \in ({\mathbb N} \cup \{0\}) \times \{\pm\}$ and $x \in P_{(n\ell,\epsilon\eta^\ell)}$, the following three conditions are equivalent.\\
(1) There exists $y \in P_{(n\ell,\epsilon\eta^\ell(-)^{k-\ell})}$ such that the equation in Figure \ref{fig:xyeqn} holds.\\
\begin{figure}[h]
\psfrag{une}{$u_{(n,\eta)}$}
\psfrag{une*}{$u_{(n,\eta)}^*$}
\psfrag{d}{$\delta^{k-\ell}$}
\psfrag{nl}{\tiny $n\ell$}
\psfrag{kml}{\tiny $k-\ell$}
\psfrag{x}{$x$}
\psfrag{y}{$y$}
\includegraphics[height=3cm]{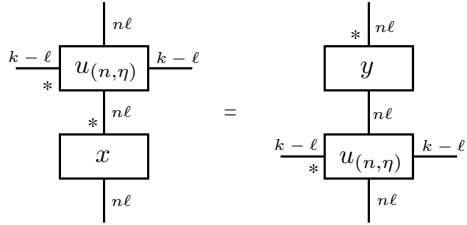}
\caption{Relation between $x \in P_{(n\ell,\epsilon\eta^\ell)}$ and  $y \in P_{(n\ell,\epsilon\eta^\ell(-)^{k-\ell})}$}
\label{fig:xyeqn}
\end{figure}

\noindent
(2) There exists $y \in P_{(n\ell,\epsilon\eta^\ell(-)^{k-\ell})}$ such that the equations in Figure \ref{fig:xyeqn2} hold.\\
\begin{figure}[h]
\psfrag{une}{\small $u_{(n,\eta)}$}
\psfrag{une*}{\small $u_{(n,\eta)}^*$}
\psfrag{x}{$x$}
\psfrag{y}{$y$}
\psfrag{d}{$\delta^{k-\ell}$}
\psfrag{nl}{\tiny $n\ell$}
\psfrag{kml}{\tiny $k-\ell$}
\includegraphics[height=3.2cm]{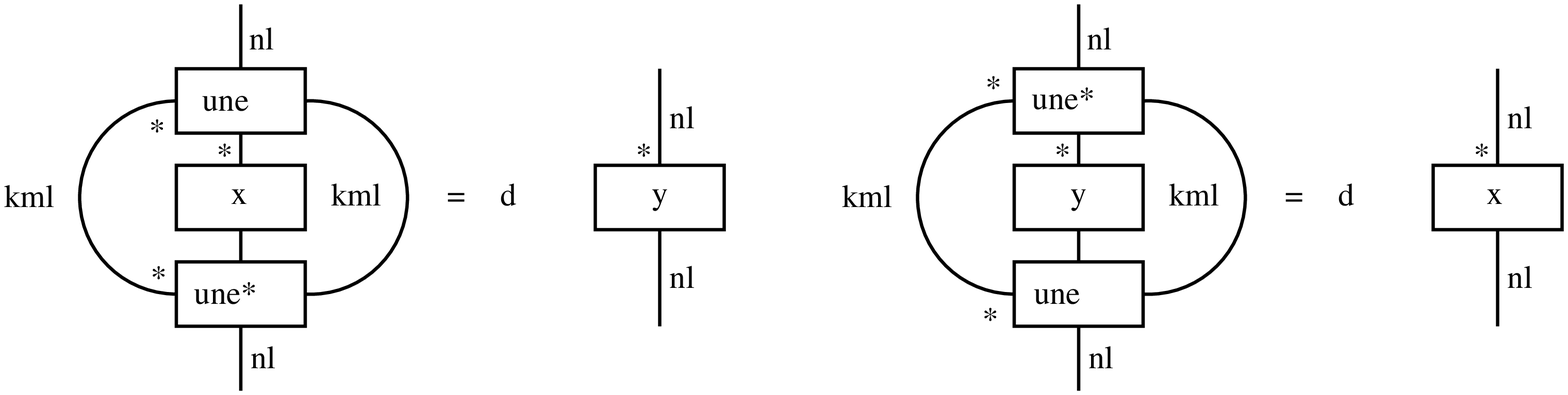}
\caption{Expressions for $x$ and $y$ in terms of each other}
\label{fig:xyeqn2}
\end{figure}

\noindent
(3) The equation in Figure \ref{fig:dcircle} holds.\\
\begin{figure}[h]
\psfrag{une}{\small $u_{(n,\eta)}$}
\psfrag{une*}{\small $u_{(n,\eta)}^*$}
\psfrag{d}{\small $\delta^{-2(k-\ell)}$}
\psfrag{nl}{\tiny $n\ell$}
\psfrag{kml}{\tiny $k-\ell$}
\psfrag{x}{$x$}
\includegraphics[height=5cm]{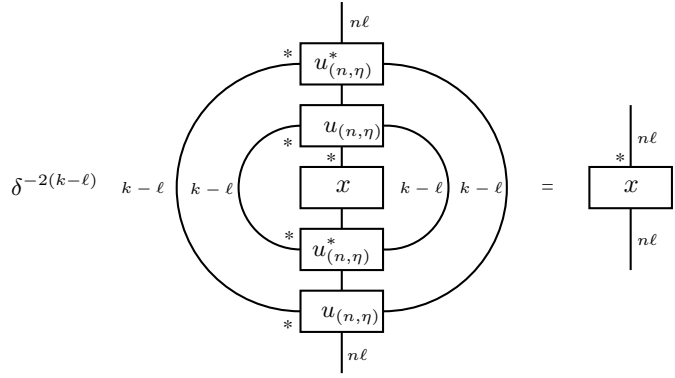}
\caption{The double circle relation}
\label{fig:dcircle}
\end{figure}

\end{proposition}

%\noindent
%If the conditions (1) and (2) hold, then, so do the equations in Figure \ref{fig:xyeqn2}.

The proof of Proposition \ref{xycircles} that we give here is an adaptation of the one in \cite{Jns1999}
with a few more details included. 
%Our proof seems to show that the assumption of sphericality of $P$ made there %(or of a weakening) is
%is not really required. 
We pave the way for the proof by defining and proving some properties of a map
from $P_{(n\ell,\epsilon\eta^\ell)}$ to $P_{(n\ell+k-\ell,\epsilon\eta^\ell)}$. Recall - see Definition 12 of \cite{KdyMrlSniSnd2019} - that a $*$-planar algebra $P$ is said to be a $C^*$-planar algebra if there exist positive normalised traces
$\tau_{\pm}: P_{(0,\pm)} \rightarrow {\mathbb C}$ such that all the traces $\tau_{\pm} \circ Z_{TR^{(0,\pm)}}$ defined on $P_{(k,\pm)}$ are faithful and positive. Thus all the spaces $P_{(k,\pm)}$  equipped with the
trace inner product are %finite-dimensional 
Hilbert spaces.

\begin{lemma}\label{lemma:sigmadef}
The map $\sigma: P_{(n\ell,\epsilon\eta^\ell)} \rightarrow P_{(n\ell+k-\ell,\epsilon\eta^\ell)}$ defined as in
the left of Figure \ref{fig:sigmadef} is an isometry with adjoint $\sigma^*$ given as in the right of Figure \ref{fig:sigmadef}.
\begin{figure}[h]
\psfrag{une}{\small $u_{(n,\eta)}$}
\psfrag{une*}{\small $u_{(n,\eta)}^*$}
\psfrag{d}{\small $\delta^{-(k-\ell)}$}
\psfrag{nl}{\tiny $n\ell$}
\psfrag{kml}{\tiny $k-\ell$}
\psfrag{x}{$x$}
\psfrag{z}{$z$}
\psfrag{u}{\small $u_{(n,\eta)}$}
\psfrag{u*}{\small $u_{(n,\eta)}^*$}
\includegraphics[height=5cm]{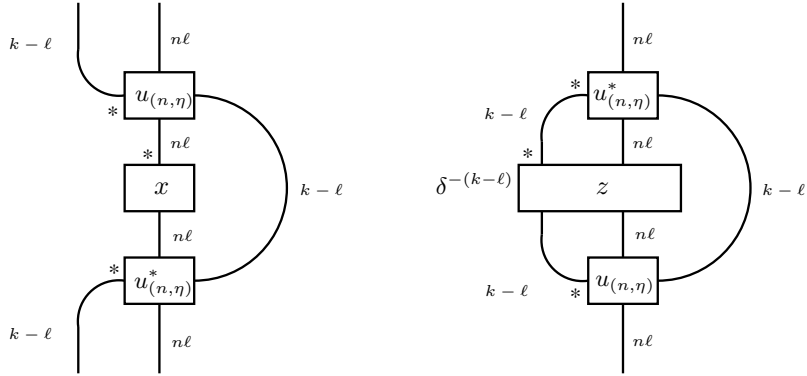}
\caption{$\sigma(x)$ and $\sigma^*(z)$ for $x \in P_{(n\ell,\epsilon\eta^\ell)}$ and $z \in P_{(n\ell+k-\ell,\epsilon\eta^\ell)}$}
\label{fig:sigmadef}
\end{figure}
Further, both $\sigma$ and $\sigma^*$ are equivariant for the $*$-operations on $P_{(n\ell,\epsilon\eta^\ell)}$ and $P_{(n\ell+k-\ell,\epsilon\eta^\ell)}$.
\end{lemma}

\begin{proof} Consider the maps $\sigma$ and $\sigma^*$ defined by the left and right side
pictures in Figure \ref{fig:sigmadef} respectively. It is clear that they are equivariant for the $*$-operations on $P_{(n\ell,\epsilon\eta^\ell)}$ and $P_{(n\ell+k-\ell,\epsilon\eta^\ell)}$. To show that they  are actually adjoints of each other, it suffices to
verify the equality of Figure \ref{fig:equality} for arbitrary $x \in P_{(n\ell,\epsilon\eta^\ell)}$ and $z \in P_{(n\ell+k-\ell,\epsilon\eta^\ell)}$.
\begin{figure}[h]
\psfrag{une}{\small $u_{(n,\eta)}$}
\psfrag{une*}{\small $u_{(n,\eta)}^*$}
\psfrag{d}{\tiny $=$}
\psfrag{nl}{\tiny $n\ell$}
\psfrag{kml}{\tiny $k-\ell$}
\psfrag{x}{$x$}
\psfrag{z}{$z^*$}
\psfrag{u}{\small $u_{(n,\eta)}$}
\psfrag{u*}{\small $u_{(n,\eta)}^*$}
\includegraphics[height=8cm]{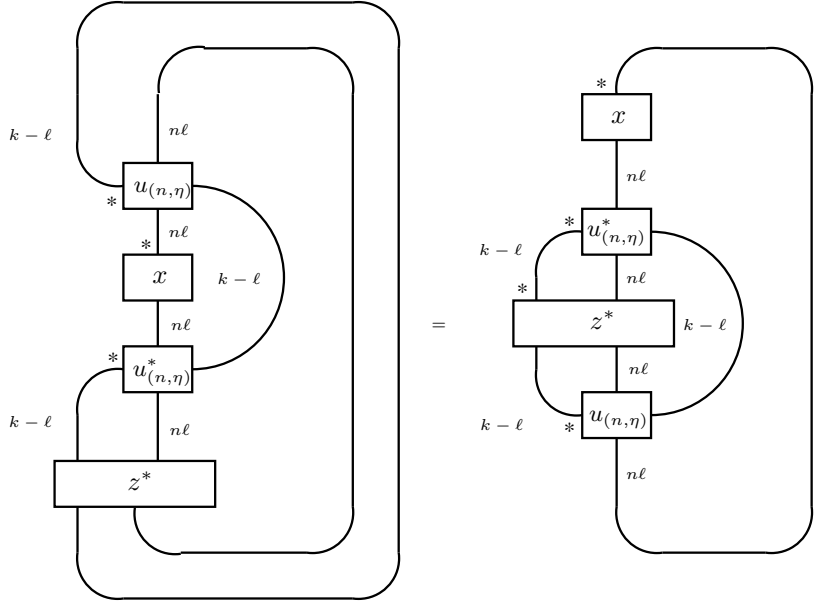}
\caption{Equality to be verified}
\label{fig:equality}
\end{figure}
This is clear by isotopy. Finally, that $\sigma^*\sigma = id$ is a simple pictorial verification using the
equalities of Figure \ref{fig:unitary}.
\end{proof}

From Lemma \ref{lemma:sigmadef} it follows that $E = \sigma\sigma^*$ is a projection onto $ran(\sigma) \subseteq 
P_{(n\ell+k-\ell,\epsilon\eta^\ell)}$ that is equivariant for the $*$-operations. Pictorially $E$ is given by the picture on the left in Figure \ref{fig:eandf}.
We will also the need the picture on the right in Figure \ref{fig:eandf} which is the projection onto the subspace $P_{(k-\ell,n\ell+k-\ell,\epsilon\eta^\ell)}$ of $P_{(n\ell+k-\ell,\epsilon\eta^\ell)}$, which also is $*$-equivariant.
\begin{figure}[h]
\psfrag{une}{\small $u_{(n,\eta)}$}
\psfrag{une*}{\small $u_{(n,\eta)}^*$}
\psfrag{d}{\small $\delta^{-(k-\ell)}$}
\psfrag{nl}{\tiny $n\ell$}
\psfrag{kml}{\tiny $k-\ell$}
\psfrag{x}{$x$}
\psfrag{z}{}
\psfrag{u}{\small $u_{(n,\eta)}$}
\psfrag{u*}{\small $u_{(n,\eta)}^*$}
\includegraphics[height=8cm]{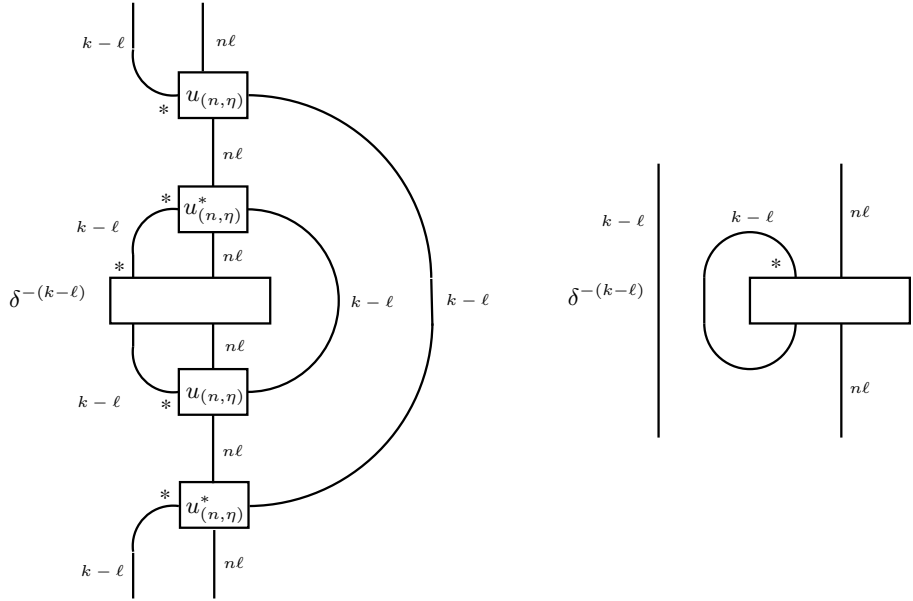}
\caption{The projections $E$ and $F$}
\label{fig:eandf}
\end{figure}

\begin{proof}[Proof of Proposition \ref{xycircles}] (1) $\Rightarrow$ (2) Using the relations 
%satisfied by the $u_{(n,\eta)}$ as 
in Figure \ref{fig:unitary}, it  is easy to see that the equation in Figure \ref{fig:xyeqn} implies those of Figure \ref{fig:xyeqn2}.\\
(2) $\Rightarrow$ (3) This is clear.\\
(3) $\Rightarrow$ (1) We need to see
that the double circle relation of Figure \ref{fig:dcircle} implies the existence of a $y$ satisfying the
equation in Figure \ref{fig:xyeqn}.

Observe that the picture on the left in Figure \ref{fig:dcircle} is given by $\sigma^*F\sigma(x)$. Thus the
double circle relation implies that $\sigma^*F\sigma(x) = x$ and hence (applying $\sigma$ on both sides and using the definition of $E$) that $EF\sigma(x) = \sigma(x)$. Since 
$E$ and $F$ are projections, norm considerations imply that 
$||EF\sigma(x)|| \leq ||F\sigma(x)|| \leq ||\sigma(x)||$. Therefore equality holds throughout and so
$F\sigma(x) = \sigma(x)$. Now define $y$ by the first equality in Figure \ref{fig:xyeqn2}. The equation $F\sigma(x) = \sigma(x)$ then implies that the
equation on the left of Figure \ref{fig:xyeqn3} holds and therefore also the equation on the right.

\begin{figure}[h]
\psfrag{une}{\small $u_{(n,\eta)}$}
\psfrag{une*}{\small $u_{(n,\eta)}^*$}
\psfrag{x}{$x$}
\psfrag{y}{$y$}
\psfrag{d}{$\delta^{k-\ell}$}
\psfrag{nl}{\tiny $n\ell$}
\psfrag{kml}{\tiny $k-\ell$}
\psfrag{u}{\small $u_{(n,\eta)}$}
\psfrag{u*}{\small $u_{(n,\eta)}^*$}
\includegraphics[height=5.8cm]{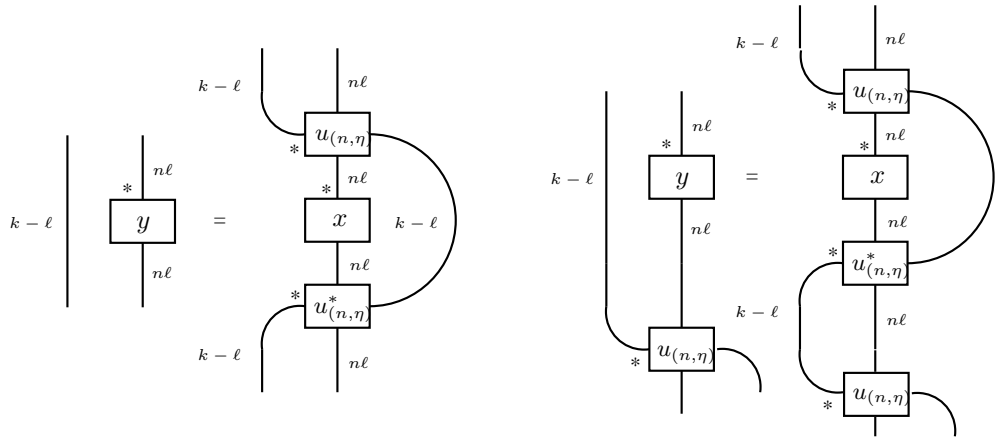}
\caption{Expressions for $x$ and $y$ in terms of each other}
\label{fig:xyeqn3}
\end{figure}

Finally, using the relations of Figure \ref{fig:unitary}, this  completes the proof.
\end{proof}

We now define subspaces $Q_{(n,\eta)}$ of $^{(\ell,\epsilon)} P_{(n,\eta)}$ by 
%\begin{lefteqn}
$Q_{(n,\eta)} = \{ x \in\,\!^{(\ell,\epsilon)}\! P_{(n,\eta)} = P_{(n\ell,\epsilon\eta^\ell)} :
\lefteqn{ \text{The equivalent conditions (1)-(3) of Proposition \ref{xycircles} hold for~} x\}}$
%\end{lefteqn}

The main result of this section is the following theorem.

\begin{theorem}{\label{deftheorem}}
The subspaces $Q_{(n,\eta)}$ 
yield a $C^*$-planar subalgebra $Q$ of $^{(\ell,\epsilon)}P$.
\end{theorem}

\begin{proof}
In order to prove that $Q$ is a $C^*$-planar subalgebra of $^{(\ell,\epsilon)}P$ it is enough to prove 
that it is a planar subalgebra of $^{(\ell,\epsilon)}P$ and that it is closed under $*$. Closure under $*$
is clear from the double circle condition of Figure \ref{fig:dcircle}.

To verify that $Q$ is a planar subalgebra of $^{(\ell,\epsilon)}P$, it suffices to see that it is closed under
the action of any set of `generating tangles'. A set of such generating tangles, albeit for the class of  `restricted tangles' - see \cite{DeKdy2018} - was given
in Theorem 3.5 of \cite{KdySnd2004}. It follows easily from that result that a set of generating tangles for all tangles is given by
$\{1^{(0,\pm)}\} \cup \{M_{(n,\eta),(n,\eta)}^{(n,\eta)}, I_{(n,\eta)}^{(n+1,\eta)}, E_{(n+1,\eta)}^{(n,\eta)}: n \geq 0\} \cup \{ R^{(n,-\eta)}_{(n,\eta)}  : n \geq 1\}$.
Here $E$, $M$ and $I$ are the tangles in Figure \ref{fig:commontangles} below.
We will show, case by case, that $Q$ is closed under the action of each of these
tangles.

\begin{figure}[h]
\centering
\psfrag{n}{\tiny $n$}
\psfrag{*}{\tiny $*$}
\psfrag{E}{\small $E_{(n+1,\eta)}^{(n,\eta)}$}
\psfrag{M}{\small $M_{(n,\eta),(n,\eta)}^{(n,\eta)}$}
\psfrag{I}{\small $I_{(n,\eta)}^{(n+1,\eta)}$}
\includegraphics[height=2.5cm]{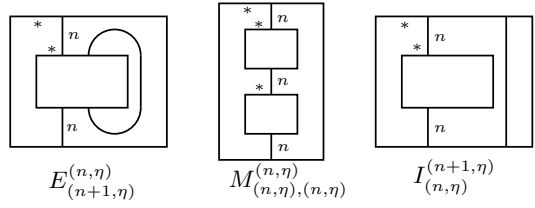}
\caption{Conditional expectation, multiplication and inclusion tangles}
\label{fig:commontangles}
\end{figure}

Closure under $1^{(0,\pm)}$: We need to check that $Z_{1^{(0,\pm)}}^{(^{(\ell,\epsilon)}P)}(1) = 1_{(0,\epsilon(\pm)^\ell)} \in Q_{(0,\pm)}$. This follows directly from the double circle relation of Figure \ref{fig:dcircle}. 

Closure under $M = M_{(n,\eta),(n,\eta)}^{(n,\eta)}$: We need to check that if $x_1,x_2 \in Q_{(n,\eta)}$, then $Z_M^{(^{(\ell,\epsilon)}P)}(x_1 \otimes x_2) = Z_{M^{(\ell,\epsilon)}}^{P}(x_1 \otimes x_2) \in Q_{(n,\eta)}$. Observe that $M^{(\ell,\epsilon)}$ is the multiplication tangle of colour $(n\ell,\epsilon\eta^\ell)$. Now suppose that
$y_1,y_2 \in P_{(n\ell,\epsilon\eta^\ell(-)^{k-\ell})}$ are such that the equation in Figure \ref{fig:xyeqn} holds for $x_1,y_1$ and for $x_2,y_2$. It is easy to see that then the 
same equation also holds for $x_1x_2,y_1y_2$.

Closure under $I = I_{(n,\eta)}^{(n+1,\eta)}$: We need to see that if $x \in Q_{(n,\eta)}$, then $Z_I^{(^{(\ell,\epsilon)}P)}(x)  = Z^P_{I^{(\ell,\epsilon)}} \in Q_{(n+1,\eta)}$. 
Observe that the tangle $I^{(\ell,\epsilon)}$ is the $\ell$-fold iterated inclusion tangle
from $P_{(n\ell,\epsilon\eta^\ell)}$ to $P_{((n+1)\ell,\epsilon\eta^\ell)}$.
Suppose that
$y \in P_{(n\ell,\epsilon\eta^\ell(-)^{k-\ell})}$ is such that the equation in Figure \ref{fig:xyeqn} holds for $x,y$. Again, an easy verification shows that the equation in Figure \ref{fig:xyeqn} also holds for $Z_{I^{(\ell,\epsilon)}}(x),Z_{\tilde{I}^{(\ell,\epsilon)}}(y)$, where $\tilde{I}^{(\ell,\epsilon)}$ is the $\ell$-fold iterated inclusion tangle
from $P_{(n\ell,\epsilon\eta^\ell(-)^{k-\ell})}$ to $P_{((n+1)\ell,\epsilon\eta^\ell(-)^{k-\ell}))}$.

Closure under $E = E_{(n+1,\eta)}^{(n,\eta)}$: We need to see that if $x \in Q_{(n+1,\eta)}$, then $Z_E^{(^{(\ell,\epsilon)}P)}(x)$ $  = Z^P_{E^{(\ell,\epsilon)}}(x) \in Q_{(n,\eta)}$. Observe that the tangle $E^{(\ell,\epsilon)}$ is the $\ell$-fold iterated conditional expectation tangle
from $P_{((n+1)\ell,\epsilon\eta^\ell)}$ to $P_{(n\ell,\epsilon\eta^\ell)}$.
Take
$y \in P_{((n+1)\ell,\epsilon\eta^\ell(-)^{k-\ell})}$ such that the equation in Figure \ref{fig:xyeqn} holds for $x,y$.
We will verify that then, the equation in Figure \ref{fig:xyeqn} also holds for $Z_{E^{(\ell,\epsilon)}}(x),Z_{\tilde{E}^{(\ell,\epsilon)}}(y)$, where $\tilde{E}^{(\ell,\epsilon)}$ is the $\ell$-fold iterated conditional expectation tangle
from $P_{((n+1)\ell,\epsilon\eta^\ell(-)^{k-\ell})}$ to $P_{(n\ell,\epsilon\eta^\ell(-)^{k-\ell}))}$.

First note that, it is an easy consequence of the relations in Figure \ref{fig:unitary} that
the relations of Figure \ref{fig:morerelations} hold for all $(n,\eta)$.
\begin{figure}[h]
\centering
\psfrag{nl}{\tiny $n\ell$}
\psfrag{*}{\tiny $*$}
\psfrag{l}{\tiny $\ell$}
\psfrag{u}{\small $u_{(n,\eta)}$}
\psfrag{u'}{\small $u_{(n+2,\eta)}$}
\psfrag{k-l}{\tiny $k-\ell$}
\psfrag{=}{\small $=$}
\includegraphics[height=5cm]{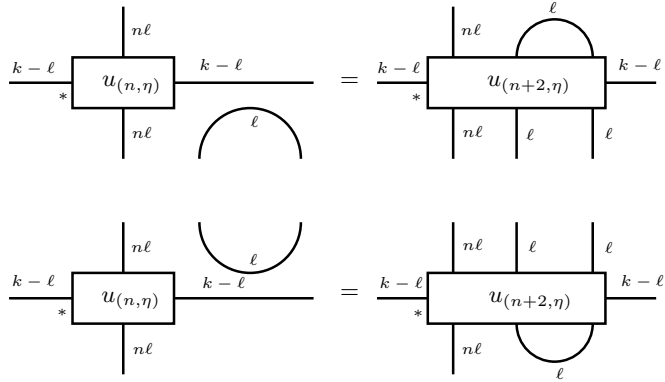}
\caption{Consequence of $\{0,\ell\}$-biunitarity relations}
\label{fig:morerelations}
\end{figure}
Now, these relations, in turn,  imply the equations in Figure \ref{fig:exprelation2}.
In this figure, the first and the third equalities follow from Figure \ref{fig:morerelations}
while the second equality is a consequence of the proof of closure under $I$.

\begin{figure}[h]
\centering
\psfrag{nl}{\tiny $n\ell$}
\psfrag{*}{\tiny $*$}
\psfrag{x}{\small $x$}
\psfrag{y}{\small $y$}
\psfrag{l}{\tiny $\ell$}
\psfrag{u}{\small $u_{(n,\eta)}$}
\psfrag{u'}{\small $u_{(n+2,\eta)}$}
\psfrag{k-l}{\tiny $k-\ell$}
\psfrag{=}{\small $=$}
\includegraphics[height=6.5cm]{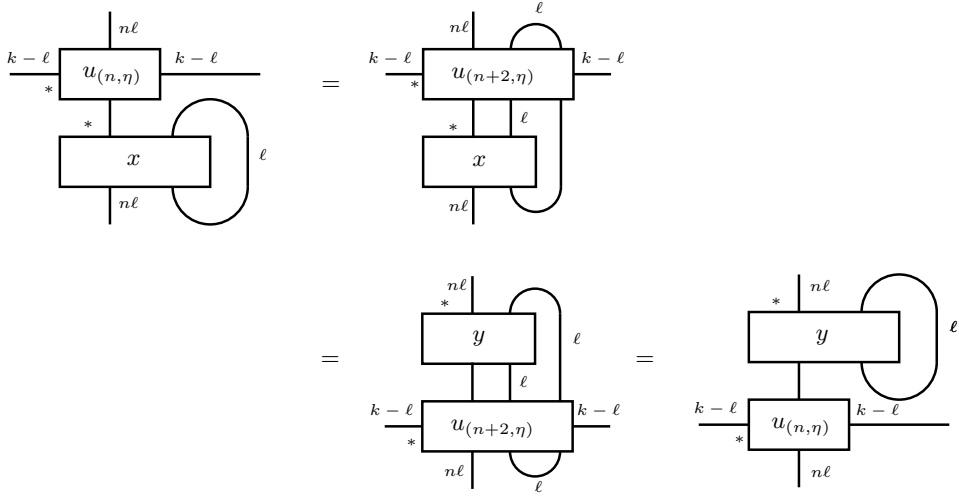}
\caption{Proof of closure under $E$}
\label{fig:exprelation2}
\end{figure}

Closure under $R = R^{(n,-\eta)}_{(n,\eta)}$: We need to see that if $x \in Q_{(n,\eta)}$, then $Z_R^{(^{(\ell,\epsilon)}P)}(x)  = Z^P_{R^{(\ell,\epsilon)}}(x) \in Q_{(n,-\eta)}$.
We illustrate how this is done when $(n,\eta) = (3,+)$. It should be clear that the proof of the general case is similar. Begin by observing that since $x \in Q_{(3,+)} \subseteq P_{(3\ell,\epsilon)}$, it satisfies the double circle relation of Figure \ref{fig:3dcx}.

\begin{figure}[h]
\centering
\psfrag{nl}{\tiny $n\ell$}
\psfrag{*}{\tiny $*$}
\psfrag{delta}{$\delta^{-2(k-\ell)}$}
\psfrag{u}{\small $u$}
\psfrag{u^*}{\small $u^*$}
\psfrag{l}{\tiny $\ell$}
\psfrag{k-l}{\tiny $k-\ell$}
\psfrag{x}{\small $x$}
\psfrag{=}{\small $=$}
\includegraphics[height=5.2cm]{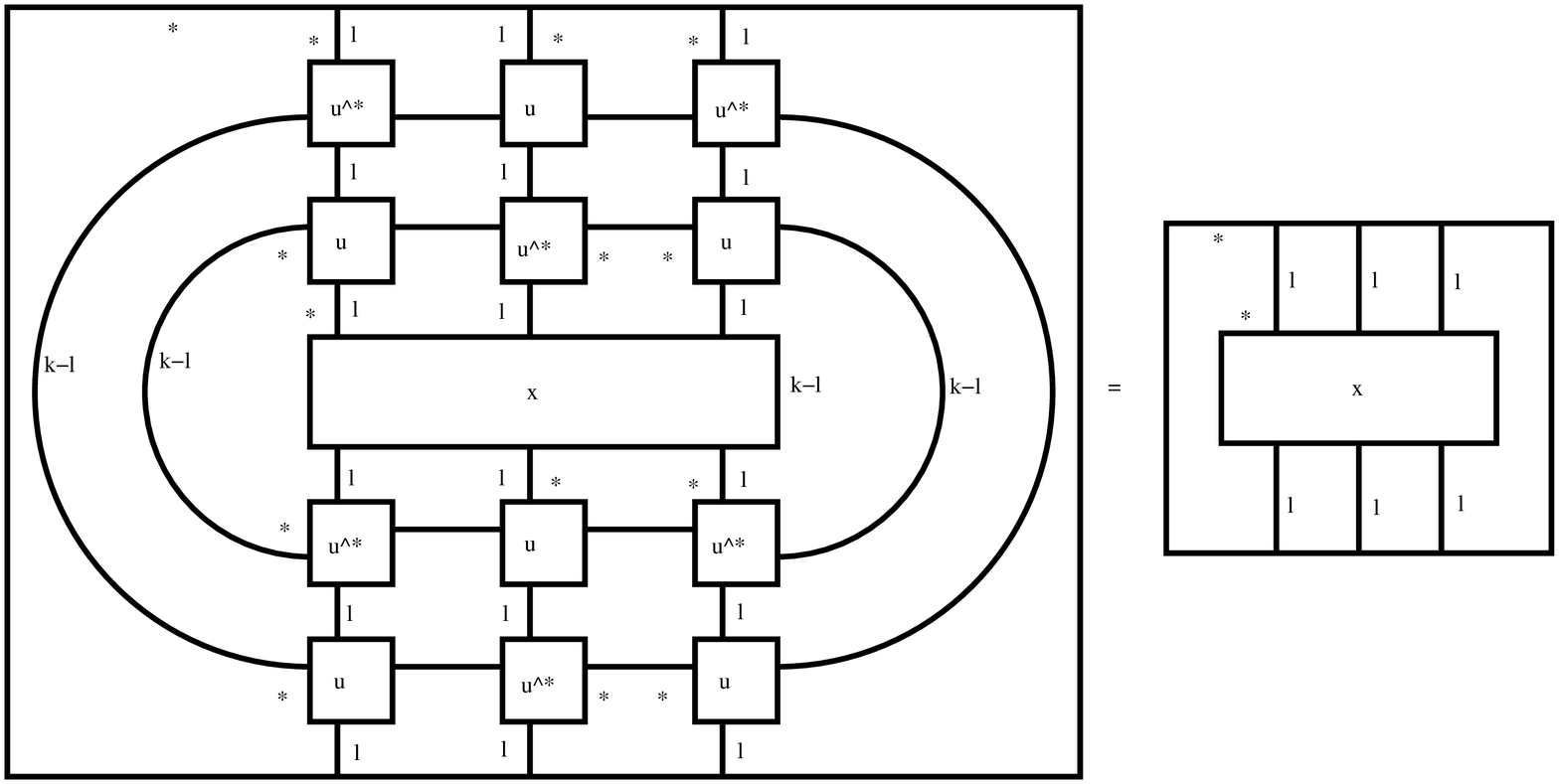}
\caption{Double circle relation for $x \in Q_{(3,+)}$}
\label{fig:3dcx}
\end{figure}

Now, moving the external $*$
$\ell$ steps counterclockwise and redrawing yields the equation in Figure \ref{fig:3dcy}.

\begin{figure}[h]
\centering
\psfrag{nl}{\tiny $n\ell$}
\psfrag{*}{\tiny $*$}
\psfrag{delta}{$\delta^{-2(k-\ell)}$}
\psfrag{u}{\small $u$}
\psfrag{u^*}{\small $u^*$}
\psfrag{l}{\tiny $\ell$}
\psfrag{k-l}{\tiny $k-\ell$}
\psfrag{x}{\small $x$}
\psfrag{=}{\small $=$}
\includegraphics[height=5.2cm]{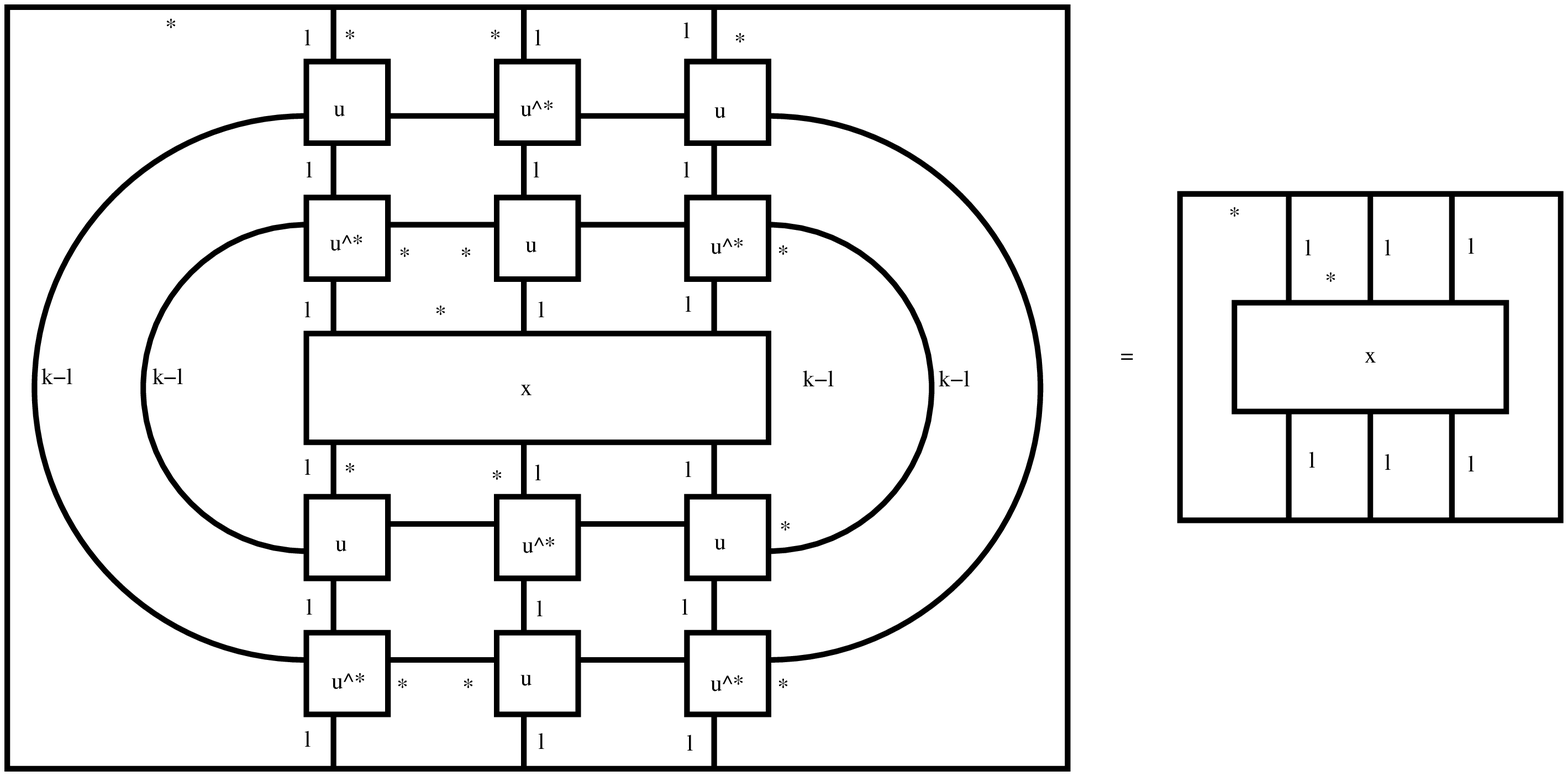}
\caption{Rotated double circle relation for $x \in Q_{(3,+)}$}
\label{fig:3dcy}
\end{figure}

A little thought now shows that this is precisely the double circle relation for
$Z^P_{R^{(\ell,\epsilon)}}(x) \in P_{(3\ell,\epsilon(-)^\ell)}$, establishing  that $Z^P_{R^{(\ell,\epsilon)}}(x)$ indeed belongs to $Q_{(3,-)}$ as desired.
\end{proof}

Next we will consider conditions under which $Q$ is a subfactor planar algebra.

\begin{proposition}\label{qsubf}
Let $P$ be the spin planar algebra on $n$ generators and $Q$ be the planar subalgebra of $^{(\ell,\epsilon)}P$ corresponding to a $\{0,\ell\}$-biunitary element $u \in P_{(k,\epsilon)}$. Then, $Q$ is a subfactor planar algebra with modulus $(\sqrt{n})^\ell$.
\end{proposition}

\begin{proof} Given Theorem \ref{deftheorem}, what remains to be seen is that
$Q$ is connected, has modulus $(\sqrt{n})^\ell$ and is spherical with positive definite picture trace. Since $P$ has modulus $\sqrt{n}$, the cabling $^{(\ell,\epsilon)}P$ has modulus
$(\sqrt{n})^\ell$ and so does $Q$. The other assertions need a little work.

Note that $Q_{(0,\eta)}$ is a subspace of $(^{(\ell,\epsilon)}P)_{(0,\eta)} = P_{(0,\epsilon\eta^\ell)}$ and so if $\epsilon\eta^\ell = +$, then $Q_{(0,\eta)}$ is 1-dimensional since $P_{(0,+)}$ is so. If $\epsilon\eta^\ell = -$, then, $Q_{(0,\eta)}$ is the
subspace of all $x \in P_{(0,-)}$ such that the double circle relation of Figure \ref{fig:dcircle} holds for $x$. From Theorem \ref{spinthm}, a basis of $P_{(0,-)}$ is given by all $S(i)$ for $i = 1,\cdots,n$ and by the black and white modulus relations, a double circle around these gives a scalar multiple of $1^{(0,-)}$. It follows that $x$ is
necessarily a scalar multiple of $1^{(0,-)}$ so that $Q_{(0,\eta)}$ is 1-dimensional, in this case as well. Hence $Q$ is connected.

To see that $Q$ is spherical, observe first that on any $P_{(n,\eta)}$ the composites of the left and right picture traces with the traces $\tau_\pm$ on
$P_{(0,\pm)}$ (which specify its $C^*$-planar algebra structure - see Definition 12 of \cite{KdyMrlSniSnd2019}) are equal. This is seen by explicit computation with the bases of $P_{(n,\eta)}$ and can be regarded as a version of sphericality for $P$. It is clear that this property descends to $Q$.

Finally observe that the picture trace on $Q_{(n,\eta)}$ is exactly the composite of
$\tau_\pm$ with the picture trace on $P_{(n\ell,\epsilon\eta^\ell)}$ and is consequently positive definite.\end{proof}

\begin{remark} In particular, Hadamard matrices, quantum Latin squares and biunitary matrices all yield subfactor planar algebras via this construction. It is not clear, however, how to get a subfactor planar algebra from a unitary error basis.
\end{remark}

\begin{remark} An even easier proof than that of Proposition \ref{qsubf} shows that if $P$ is a subfactor planar algebra and $Q$ is the planar subalgebra of $^{(\ell,\epsilon)}P$ corresponding to a $\{0,\ell\}$-biunitary element $u \in P_{(k,\epsilon)}$, then, $Q$ is a subfactor planar algebra with modulus $(\sqrt{n})^\ell$.

\end{remark}

In case $\ell =1$, the planar algebra $Q$ is even irreducible.

\begin{proposition}
Let $P$ be the spin planar algebra on $n$ generators and $Q$ be the planar subalgebra of $^{(1,\epsilon)}P$ corresponding to a $\{0,1\}$-biunitary element $u \in P_{(k,\epsilon)}$. Then, $Q$ is an irreducible subfactor planar algebra.
\end{proposition}

\begin{proof} Only the irreducibility of $Q$ needs to be seen and we will show using explicit bases computations that $dim(Q_{(1,+)}) = 1$. 
We only consider the $\epsilon = +$ case, the proof in the other case being similar. Thus $Q_{(1,+)} \subseteq P_{(1,+)}$.
Begin with $x = \sum_i \lambda_i e(i] \in Q_{(1,+)}$. The double circle relation for $x$
implies that the equation of Figure \ref{fig:irredproof} holds.
\begin{figure}[h]
\psfrag{une}{\small $u$}
\psfrag{une*}{\small $u^*$}
\psfrag{d}{\small $\sum_i \lambda_i (\sqrt{n})^{-2(k-1)}$}
\psfrag{nl}{\tiny $$}
\psfrag{kml}{\tiny $k-1$}
\psfrag{x}{$x$}
\psfrag{s1}{\tiny $s_i$}
\includegraphics[height=5cm]{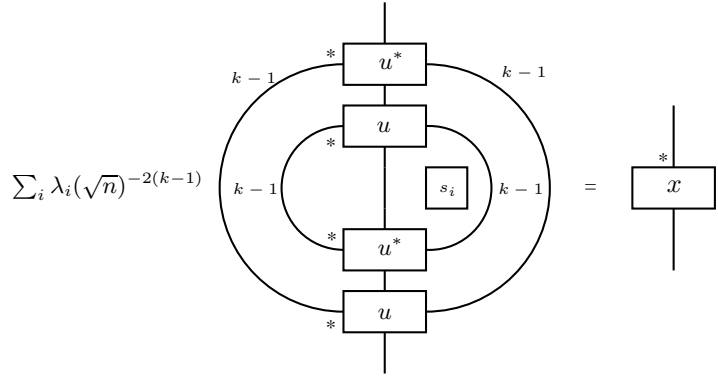}
\caption{Equation satisfied by $x$}
\label{fig:irredproof}
\end{figure}

Now, using the biunitarity relations of Figure \ref{fig:unitary} together with the black and white modulus relations, the left hand side of 
Figure \ref{fig:irredproof} simplifies to $\frac{1}{n}\sum_i \lambda_i 1_{(1,+)}$, finishing the proof.\end{proof}

\begin{remark}\label{exception} The proof of Theorem \ref{deftheorem} relies heavily on 
Proposition \ref{xycircles}, and in particular, the double circle relation,  which uses the assumption that $P$ is a spherical $C^*$-planar algebra. However, even if $P$ is just a $*$-planar algebra, without the positivity conditions or sphericality holding, $Q$ can still be shown to be a $*$-planar subalgebra of $P$.
The proof is a little longer using a different larger set of generating tangles.
\end{remark}

\section{The planar algebra associated to a finite group Latin square}

Throughout this section, $G$ will be a finite group of order $n$. Associated to $G$
is its multiplication table which is an $n \times n$ Latin square and consequently yields - see Example \ref{lstoqls} - a quantum Latin square of size $n$. If $P$ is the
spin planar algebra on $n$ generators, this quantum Latin square gives, by Theorem \ref{spinquantumthm}(2), a $\{0,1\}$-biunitary element in $P_{(3,+)}$. Applying the construction of Theorem \ref{deftheorem} to this biunitary element, we get an irreducible subfactor planar algebra. The main result of this section identifies this
planar algebra with the well known group planar algebra $P(G)$ - see \cite{Lnd2002}.

We begin with a notational convention. The generating set for the spin planar algebra is the underlying set of the group $G$ and we will use notation such as
$e[g)^{h_1 \cdots h_{m}}_{k_1 \cdots k_m}(\ell]$ for $g,h_i,k_j,\ell \in G$ to denote basis elements of $P$. With this notation, the $\{0,1\}$-biunitary element $u \in P_{(3,+)}$ corresponding to the multiplication table of the group $G$ is seen to be given by
$$
u = \sum_{h,k} e^{kh}_k(h] \in P_{(3,+)},
$$
according to Theorem \ref{spinquantumthm}(2).

Let $Q$ be the planar subalgebra of $P$ corresponding to the biunitary element $u$
as in Theorem \ref{deftheorem}. The next proposition identifies $Q$ and we sketch a proof leaving out most of the computational details.

\begin{proposition}
The planar algebra $Q$ is isomorphic to $P(G)$.
\end{proposition}

\begin{proof}
The planar algebra  $P(G)$ has a presentation by generators and relations.
We show that it is isomorphic to $Q$ in a series of steps - (1) computing the dimensions of the spaces of $Q$ and observing that these are equal to those of the spaces of $P(G)$, (2) by specifying a map from the universal planar algebra on $L = L_{(2,+)} = G$ to $Q$ and checking that the relations hold in $Q$, thereby yielding a planar algebra map
from $P(G)$ to $Q$ and (3) verifying that this map is surjective.

Step 1: We first observe that for $ m \geq 1$, the elements $u_{(2m,+)}$ and $u_{(2m+1,+)}$ are given explicitly by:
\begin{eqnarray*}
u_{(2m,+)} &=& \sum_{g,h_1,\cdots,h_m} e^{g,h_1^{-1},\cdots,h_m^{-1}}_{gh_1,gh_2,\cdots,gh_m,g}\\
u_{(2m+1,+)} &=& \sum_{g,h_1,\cdots,h_{m+1}} e^{g,h_1^{-1},\cdots,h_m^{-1}}_{gh_1,gh_2,\cdots,gh_{m+1}}(h_{m+1}^{-1}]
\end{eqnarray*}
Now consider elements $x,y \in P_{(2m,+)}$ given by:
\begin{eqnarray*}
x &=& \sum_{k_1,\cdots,k_m,\ell_1,\cdots,\ell_m} \alpha_{(k_1,\cdots,k_m,\ell_1,\cdots,\ell_m)}e^{k_1,\cdots,k_m}_{\ell_1,\cdots,\ell_m}\\
y &=& \sum_{k_1,\cdots,k_m,\ell_1,\cdots,\ell_m} \beta_{(k_1,\cdots,k_m,\ell_1,\cdots,\ell_m)}e^{k_1,\cdots,k_m}_{\ell_1,\cdots,\ell_m},
\end{eqnarray*}
for $\alpha_{(k_1,\cdots,k_m,\ell_1,\cdots,\ell_m)}, \beta_{(k_1,\cdots,k_m,\ell_1,\cdots,\ell_m)} \in {\mathbb C}$. The condition that $x,y$ satisfy the condition in Figure \ref{fig:xyeqn} is seen to imply that for all $g \in G$,  $$\alpha_{(k_1,\cdots,k_m,\ell_1,\cdots,\ell_m)} = \alpha_{(gk_1,\cdots,gk_m,g\ell_1,\cdots,g\ell_m)}.$$
Conversely, if this condition holds, setting $\beta_{(k_1,\cdots,k_m,\ell_1,\cdots,\ell_m)} = \alpha_{(k_1^{-1},\cdots,k_m^{-1},\ell_1^{-1},\cdots,\ell_m^{-1})}$, the elements $x$ and $y$ are checked satisfy the condition in Figure \ref{fig:xyeqn}.

Thus, a basis of $Q_{(2m,+)}$ is given by the set of all
$$
\sum_{g \in G} e^{gk_1,\cdots,gk_m}_{g\ell_1,\cdots,g\ell_m}
$$
as $(k_1,\cdots,k_m,\ell_1,\cdots,\ell_m)$ vary over the representatives of the diagonal action of $G$ on $G^{2m}$.
It follows that the dimension of $Q_{(2m,+)}$ is given by $n^{2m-1}$ and a similar proof shows that the dimension of $Q_{(2m+1,+)}$ is given by $n^{2m}$.

Step 2: Next we define a map from the universal planar algebra on the label set $L = L_{(2,+)} = G$ to $Q$ given by sending $g \in G$ to
$$
X_g = \sum_{q \in G} e^q_{qg} \in Q_{(2,+)}.
$$
A long but routine verification - which we omit entirely - establishes that all relations satisfied by the $g \in G$ in the planar algebra $P(G)$ also hold for their images $X_g$ in $Q$ thus giving a planar algebra map from $P(G)$ to $Q$.

\begin{figure}[!h]
\psfrag{k1}{\tiny $k_1$}
\psfrag{k2}{\tiny $k_2$}
\psfrag{k3}{\tiny $k_3$}
\psfrag{kmm1}{\tiny $k_{m-1}$}
\psfrag{km}{\tiny $k_m$}
\psfrag{l1}{\tiny $\ell_1$}
\psfrag{l2}{\tiny $\ell_2$}
\psfrag{l3}{\tiny $\ell_3$}
\psfrag{lmm1}{\tiny $\ell_{m-1}$}
\psfrag{lm}{\tiny $\ell_m$}
\psfrag{cdots}{$\cdots$}
\includegraphics[height=4cm]{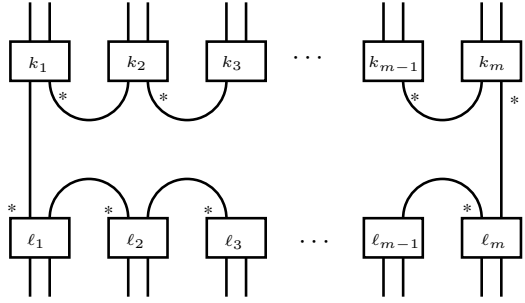}
\caption{An element of $P_{(2m,+)}$}
\label{fig:basiselt}
\end{figure}

Step 3: To verify that $P(G) \rightarrow Q$ is surjective, it suffices to see that 
$P(G)_{(2m,+)} \rightarrow Q)_{(2m,+)}$ is surjective for all $m > 0$. A routine calculation shows that the element in Figure \ref{fig:basiselt} goes to the basis
element $\sum_{g \in G} e^{gk_1,\cdots,gk_m}_{g\ell_1,\cdots,g\ell_m}$ of $Q_{(2m,+)}$, finishing the proof.
\end{proof}

\end{document}